\documentclass[12pt,reqno, a4paper]{amsart}
\usepackage{graphicx}
\usepackage{fancyhdr}
\usepackage{enumerate}
\usepackage{amsmath}
\usepackage{amsthm}
\usepackage{mathabx}
\usepackage{amssymb}
\usepackage{relsize}
\usepackage{pdfsync}
\usepackage[colorlinks=true,linkcolor=blue,citecolor=magenta]{hyperref}

\usepackage[normalem]{ulem}

\usepackage{tikz}
\usetikzlibrary{decorations.pathmorphing,patterns}
\usetikzlibrary{calc,arrows}

\usepackage{MnSymbol}

\usepackage{xcolor}

\usepackage{changebar}
\usepackage{pdfsync}

\makeindex
\newtheorem{theorem}{Theorem}[section]
\newtheorem{lemma}[theorem]{Lemma}
\newtheorem{proposition}[theorem]{Proposition}
\newtheorem{corollary}[theorem]{Corollary}

\newtheorem{remark}[theorem]{Remark}

\newcommand{\al}{\alpha}
\newcommand{\Om}{\Omega}
\newcommand{\cal}{\mathcal}
\newcommand{\mc}[1]{{\mathcal #1}}

\newcommand{\bb}[1]{{\mathbb #1}}

\newcommand\bbE{{\mathbb E}}

\newcommand\cX{{\mathcal X}}

\newcommand\R{{\mathbb R}}

\newcommand\T{{\mathbb I}}

\newcommand \ga{\gamma}

\newcommand{\dd  }{\mathrm{d}}

\renewcommand{\tilde}{\widetilde}
\renewcommand{\bar}{\overline}
\renewcommand{\le}{\leqslant}
\renewcommand{\leq}{\leqslant}
\renewcommand{\ge}{\geqslant}

\numberwithin{equation}{section}

\newcommand{\mm}[1]{{#1}}
\newcommand{\so}[1]{{#1}}

\begin{document}

\title{An open microscopic model of heat conduction: evolution and non-equilibrium stationary states}

\author{Tomasz Komorowski}
\address{Tomasz Komorowski: Institute of Mathematics, Polish Academy
  Of Sciences\\Warsaw, Poland.} 
\email{{\tt komorow@hektor.umcs.lublin.pl}}

\author{Stefano Olla}
\address{Stefano Olla: CNRS, CEREMADE\\
  Universit\'e Paris-Dauphine, PSL Research University\\
  75016 Paris, France }
  \email{olla@ceremade.dauphine.fr}

\author{Marielle Simon}
\address{Marielle Simon: Inria, Univ. Lille, CNRS, UMR 8524, Laboratoire Paul Painlevé, F-59000 Lille, France}
\email{marielle.simon@inria.fr}


\begin{abstract}
  We consider a one-dimensional chain of coupled oscillators in contact at both ends
    with heat baths at different temperatures, and subject to an external force at one end.
    The Hamiltonian dynamics in the bulk is perturbed by random exchanges of the neighbouring
    momenta  such that  the energy is locally conserved.
 We prove that in the stationary state the energy and the volume stretch profiles,
      in large scale limit, converge to the solutions of a diffusive system with Dirichlet boundary conditions.
      As a consequence the macroscopic temperature stationary profile presents a maximum inside the
      chain higher than the thermostats temperatures,
      as well as the possibility of uphill diffusion (energy current against the temperature gradient). \mm{Finally, we are also able to derive the non-stationary macroscopic coupled diffusive equations followed by the energy and volume stretch profiles.}
\end{abstract}

\thanks{This work was partially supported by the grant 346300 for
  IMPAN from the Simons Foundation and the matching 2015-2019 Polish MNiSW fund,
  and by the ANR-15-CE40-0020-01 grant LSD.
  T.K.  acknowledges the support of the National Science Centre:
NCN grant 2016/23/B/ST1/00492. M. S. thanks Labex CEMPI (ANR-11-LABX-0007-01) and  the project EDNHS ANR-14-CE25-0011 of the French National Research Agency (ANR) for their support. This project has received funding from the European Research Council (ERC) under  the European Union's Horizon 2020 research and innovative programme (grant agreement  n\textsuperscript{o}715734)} 

\maketitle


\section{Introduction}
\label{sec:intro}

Non-equilibrium transport in one dimension presents itself to be an interesting phenomenon
and in many models numerical simulations can be easily performed.
Most of the attention has been focused on the study
of the non-equilibrium stationary states (NESS),
where the systems are subject to exterior heat baths at
different temperatures and other external forces, so that
the invariant measure is not the equilibrium Gibbs measure.

The most interesting models are those with various conserved quantities
(energy, momentum, volume stretch...) whose transport is coupled. The
densities of these quantities
may evolve at different time scales,  particularly when the spatial dimension of the system
equals one. For example, in the  Fermi-Pasta-Ulam (FPU) chain, volume stretch, mechanical energy and momentum
all evolve in the hyperbolic time scale. Their evolution is governed by the
Euler equations (see \cite{oe}) while the thermal energy is expected to evolve at a superdiffusive time scale,
with an autonomous evolution described by a fractional heat equation.  This has been predicted \cite{Sp},   confirmed
by many numerical experiments on the NESS \cite{llp97, llp03} 
and proved analytically for harmonic chains with
random exchanges of momenta that conserve energy,
  momentum and volume stretch, see \cite{jko}.

\mm{In contrast to the situation described above, the
present paper deals with a system} for which conserved quantities evolve macroscopically
in the same \emph{diffusive} time scale,
and their macroscopic evolution is governed by a system of
\emph{coupled} diffusive equations. One example is given by the chain of coupled rotors, whose dynamics conserves the energy and the angular momentum.
In \cite{iaco} the NESS of  this chain is studied numerically, when Langevin thermostats
are applied at both ends, while a constant force is applied to one end and the position of the rotor
on the opposite side is kept fixed. While heat flows from the thermostats, work is performed by
the torque, increasing the mechanical energy, which is then transformed into thermal energy by the
dynamics of the rotors. The stationary temperature profiles observed numerically in \cite{iaco} present
a maximum inside the chain higher than  the temperature of both thermostats.
Furthermore, a negative linear response for the energy flux has been observed for certain values
of the external parameters. \mm{This phenomenon is
referred to in the literature as  an {\em uphill diffusion},
see \cite{krihsna} or \cite{cmp} and references therein}.
These numerical results have been confirmed in \cite{illp14}, as well
as an instability of the system when thermostats are at zero temperature.

The present work aims at describing a similar
phenomenon for  the NESS, but for a different model.
In particular, we are able to show rigorously that the maximum of the temperature profile occurs inside the system.
According to our knowledge it is the first theoretical result \mm{that rigorously establishes the heating inside the system and uphill diffusion phenomena. }

More specifically, we consider a chain of unpinned harmonic oscillators whose dynamics is perturbed by a random
mechanism that conserves the energy and volume stretch: any two nearest neighbour
particles exchange their momenta randomly in such a way that the
total kinetic energy is conserved.
Two Langevin thermostats are attached at the opposite ends of the chain and a
constant force $\bar\tau_+$ acts on the last particle of the chain.
This system has only two conserved quantities: total energy and volume.
Since the random mechanism does not conserve the total momentum,
the macroscopic behaviour of these two quantities is diffusive, and the
non stationary hydrodynamic limit with periodic boundary conditions
(no thermostats or exterior force present) has been proven in \cite{kos1}.

The action of this constant force puts the system out of equilibrium,
even when the temperatures of the thermostats are equal.
As in the rotor chain described above, the {exterior} force performs positive work on the system,
that increases the mechanical energy (concentrated on low  frequency modes).
The random mechanism, \mm{which consists in the kinetic energy exchange between neighbouring atoms,
see definition \eqref{eq:generator} and the following explanations, }
transforms the mechanical energy into the thermal one
(uniformly distributed in all frequencies, when the
system is in a local equilibrium),
which is eventually dissipated by the thermostats.
This transfer of mechanical into thermal energy happens in the bulk of the system
and is already completely predicted by the solution of the
macroscopic diffusive system of equations obtained in the hydrodynamic limit \cite{kos1},
\mm{see also \cite{bc} for a similar model without boundary conditions.}

In the present article we study the NESS of this dynamics. We prove, \mm{see 
Theorem \ref{theo:energy} below},
that the energy and the
volume stretch profiles converge to the stationary solution of the diffusive system,
with the boundary conditions  imposed by the thermostats and the external tension.
It turns out that these stationary equations can be solved explicitly and the stretch profile is linear
between $0$ and $\bar\tau_+$, while the thermal energy (temperature)
profile is a concave parabola with the boundary
conditions coinciding with the temperatures of the thermostats.
The curvature of the parabola is proportional to $\bar\tau_+^2$, i.e.~the increase of the bulk temperature
is not a linear response term. In the case $\bar\tau_+= 0$, the NESS was studied in \cite{bo1},
where the temperature profile is  proved to be linear: \mm{more details are available in \cite{bcpre,bllk}}.
 This \emph{heating inside the system} phenomenon is similar to the ohmic loss, due to the diffusion of electricity
in a resistive system (see e.g.~\cite{bh08}).

The NESS for our model also provides a simple example of an \emph{uphill} energy diffusion:
if the force $\bar\tau_+$ is large enough and applied on the side, where the coldest thermostat is acting,
the sign of the energy current can be equal to the one of the temperature gradient,  \mm{see 
Theorem \ref{theo:current} below}. It is not surprising after
understanding that this is regulated by a system of two diffusive coupled equations.
On the other hand, the model does not work as a \emph{stationary refrigerator}:
i.e.~a system where the heat on the coldest
thermostat  flows into it.

Our results suggest that there is a universal behaviour of the temperature profiles in the NESS
when there are at least two conserved quantities. 
This should be tested on a system with
three conserved quantities that evolves in the diffusive scale, such as e.g.~a  non-acoustic harmonic chain
with a random exchange of  momentum as considered in \cite{kona},
where the non-stationary hydrodynamic limit is proven.
\mm{An attempt to describe more generally the  systems for which the phenomena
of an uphill energy diffusion and  heating inside  the system occur is  made in \cite{O19}.}

\medskip

\mm{Let us add a comment on the proofs of our main results.
In proving Theorem  \ref{theo:energy} (on the asymptotics of energy and stretch profile)
we need to make an additional  hypothesis concerning
the strength $\ga>0$ of the noisy part of the dynamics,
see \eqref{eq:qdynamics} and \eqref{eq:pdyn}. More precisely we suppose that $\ga=1$.
This assumption is of purely technical nature and allows to discard the term corresponding
to the equipartition of the mechanical and kinetic energy in the decomposition \eqref{energy:decomp}
of the microscopic energy profile of  the chain (the term $\mathcal{H}_n^{\text{m}}$).
We conjecture that this term vanishes in the stationary macroscopic limit,
making the conclusion of the theorem valid for any $\ga>0$,
but are unable to prove this fact at the moment.
We do not need this hypothesis in our proof of Theorem  \ref{theo:current} (on the uphill diffusion phenomenon).}

\so{In Appendix \ref{sec:form-deriv-equat} we give a proof for the non-stationary macroscopic evolution
  of the energy and the volume stretch profiles in the diffusive space-time scaling. As for the NESS, the proof is rigorous only for $\gamma = 1$, for similar reasons. The corresponding result with periodic boundary conditions was contained in \cite{bc}.} 

\medskip

The rest of the paper is organized as follows: in Section \ref{sec:adiab-micr-dynam} we define the microscopic model under investigation and give the expected macroscopic system of equations, showing the phenomenon of uphill diffusion. In Section \ref{sec:result} we state the main  results of the paper, namely the convergence of the non-equilibrium stationary profiles of elongation, current and energy. In order to prove them, we need precise computations on the averages and second order moments taken with respect to the NESS. Section \ref{sec:stat} provides elements of the proofs and preliminary computations on the averages, while Section \ref{sec:moment} provides all the remaining technical lemmas, concerning the second order moments.



\section{Microscopic dynamics and macroscopic behaviour}
\label{sec:adiab-micr-dynam}

\subsection{Open chain of oscillators}
\label{sec2.1}

Let $\T_n := \{1,\dots,n\}$, $\bar\T_n:=\T_n \cup \{0\}$ and $\T := [0,1]$. 
The configuration space $ \Omega_n:=\R^{\T_n}\times\R^{\bar\T_n}$  consists of all sequences
$(\mathbf{q},\mathbf{p}):=\{q_x,p_x\}_{x \in \bar\T_n}$, where
$p_x\in\R$ stands for the momentum of the oscillator at site $x$, and
$q_x\in\R$ represents its position.  The interaction
between two particles $x$ and $x+1$ is described by the quadratic
potential energy $V(q_x-q_{x+1}):=\frac12 (q_x-q_{x+1})^2$ of a
harmonic spring {linking} the particles.  At the boundaries the system
is connected to two Langevin heat baths at temperatures $T_-$ and
$T_+$. Furthermore, on the right boundary acts a force (tension)
$\bar\tau_+$, possibly slowly changing in time at a scale
$t/n^2$. Note that the system is {\em unpinned}, i.e.~there is no external
potential binding the particles. Consequently, the absolute positions $q_x$ 
do not have a precise meaning, and the dynamics  only depends on the interparticle elongations 
$r_x: = q_x - q_{x-1}, x\in\T_n$. The configurations can then be described by
 \begin{equation}
  \label{Omn}
 (\mathbf r, \mathbf p) = (r_1, \dots, r_n, p_0, \dots, p_n) \in \Omega_n. 
\end{equation}
The total energy of the system is defined by the Hamiltonian:
\begin{equation}
\label{Hn}
\mathcal{H}_n (\mathbf r, \mathbf p) :=\sum_{x\in\T_n} {\cal E}_x+ \frac{p_0^2}2,
\end{equation}
 with
\begin{equation*}
{\cal E}_x:= \frac{p_x^2}2 + \frac{r_{x}^2}2,\quad x\in\T_n.
\end{equation*}
 \mm{We investigate this system in the diffusive time scale (when the ratio of the  microscopic \textit{vs} macroscopic time is $n^2$), therefore} the equations of the  microscopic dynamics are given in the bulk by 
\begin{align}
\label{eq:qdynamics}
    \dd   r_x(t) &= n^2 \left(p_x(t)- p_{x-1}(t)\right)  \dd   t, \qquad  x\in\T_n \vphantom{\Big(}\\
    \dd   p_x(t) &= n^2 \left(r_{x+1}(t)- r_x(t)\right)  \dd   t -  \gamma n^2  p_x(t)
     \dd  t \notag \vphantom{\Big(} \\
     & \quad + n\sqrt{\gamma}\big(p_{x-1}(t)\dd
       w_{x-1,x}(t)-p_{x+1}(t)\dd w_{x,x+1}(t)\big), \quad 
       x\in \{1, \dots, n-1\} \vphantom{\Big(}\label{eq:pdyn}\end{align}
 and at the boundaries: 
    \begin{align}
     \dd   p_0(t) &= n^2 \; r_1(t)   \dd   t
     -  \frac{n^2}2(\gamma+\tilde \gamma)p_0(t) \dd t   - n \sqrt \gamma p_1(t) \dd w_{0,1}(t) + n\sqrt{\tilde\gamma T_-} \dd \tilde w_0(t) \vphantom{\Big(} \label{eq:rbd}\\
 \dd   p_n(t) &= -n^2 \; r_n(t)  \dd   t  + n^2\; \bar\tau_+(t)  \dd t - \frac{n^2}2(\gamma +  \tilde \gamma)p_n(t) \dd t  \vphantom{\Big(} \notag \\
                  & \qquad \qquad \qquad \qquad \qquad \quad
                    + n \sqrt \gamma p_{n-1}(t)\dd w_{n-1,n}(t)
                    + n\sqrt{\tilde\gamma T_+} \dd \tilde w_n(t)\vphantom{\Big(} \label{eq:pbd}
  \end{align}
where $w_{x,x+1}(t)$, $x\in\{0,\dots,n-1\}$, $\tilde w_0(t)$ and $\tilde
w_n(t)$ are independent, standard one dimensional Wiener processes, and $\gamma>0$ (resp. $\tilde\gamma >0$) regulates the intensity of the random perturbation (resp. the Langevin thermostats).  See Figure \ref{fig:osci} for a representation of the chain.
\mm{Note that the purely Hamiltonian dynamics is perturbed
  by a stochastic noise which exchanges kinetic energy between the neighbouring atoms,
  and with the boundary thermostats}.

\begin{figure}

\begin{center}
\begin{tikzpicture}
\node[circle,fill=black,inner sep=1.2mm] (e) at (0,0) {};
\node[circle,fill=black,inner sep=1.2mm] (f) at (2,0) {};
\node[circle,fill=black,inner sep=1.2mm] (g) at (3.5,0) {};
\node[circle,fill=black,inner sep=1.2mm] (h) at (4.8,0) {};
\node[circle,fill=black,inner sep=1.2mm] (i) at (6.8,0) {};
\node[circle,fill=black,inner sep=1.2mm] (j) at (8.5,0) {};
\node[circle,fill=black,inner sep=1.2mm] (k) at (10,0) {};
\node[circle,fill=black,inner sep=1.2mm] (l) at (11.3,0) {};

\draw[dashed] (2,0) -- (3.5,0);
\draw[dashed] (8.5,0) -- (10,0);
\draw[ultra thick, blue, ->] (11.3,0) -- (12.5,0);

\draw[thick, <->] (3.5,-1.2) -- (4.8,-1.2);

\draw (0, -0.6) node[] {$q_{0}$};
\draw (2, -0.6) node[] {$q_{1}$};
\draw (11.3, -0.6) node[] {$q_{n}$};
\draw (3.5, -0.6) node[] {$q_{x-1}$};
\draw (4.8, -0.6) node[] {$q_{x}$};
\draw (6.8, -0.6) node[] {$q_{x+1}$};
\draw[dashed] (3.5,-1.5) -- (3.5,-1);
\draw (4.1, -1.5) node[] {$r_x$};
\draw[dashed] (4.8,-1.5) -- (4.8,-1);

\draw (-0.6,1.8) node[] {\large\color{red}$T_-$};
\draw (12,1.8) node[] {\large\color{red}$T_+$};
\draw (12.4,-0.4) node[] {\color{blue}$\bar\tau_+(t)$};

\draw[decoration={aspect=0.3, segment length=3mm, amplitude=3mm,coil},decorate] (0,0) -- (2,0); 
\draw[decoration={aspect=0.3, segment length=1.8mm, amplitude=3mm,coil},decorate] (3.5,0) -- (4.9,0); 
\draw[decoration={aspect=0.3, segment length=3mm, amplitude=3mm,coil},decorate] (4.8,0) -- (6.9,0); 
\draw[decoration={aspect=0.3, segment length=2.5mm, amplitude=3mm,coil},decorate] (6.8,0) -- (8.6,0); 
\draw[decoration={aspect=0.3, segment length=1.8mm, amplitude=3mm,coil},decorate] (10,0) -- (11.4,0); 

\fill [pattern = north east lines, pattern color=red] (-0.3,0.8) rectangle (0.3,2);
\fill [pattern = north east lines, pattern color=red] (11,0.8) rectangle (11.6,2);
\node (c) at (-0.3,1.5) {};
\node (d) at (-0.1,0.1) {};
\node (a) at (11.6,1.5) {};
\node (b) at (11.4,0.1) {};

\draw (c) edge[dashed, ultra thick, red, ->, >=latex, bend right=60] (d);
\draw (a) edge[dashed, ultra thick, red, ->, >=latex, bend left=60] (b);

\end{tikzpicture}
\end{center}

\caption{Oscillator chains with heat baths and one boundary force.}
\label{fig:osci}
\end{figure}
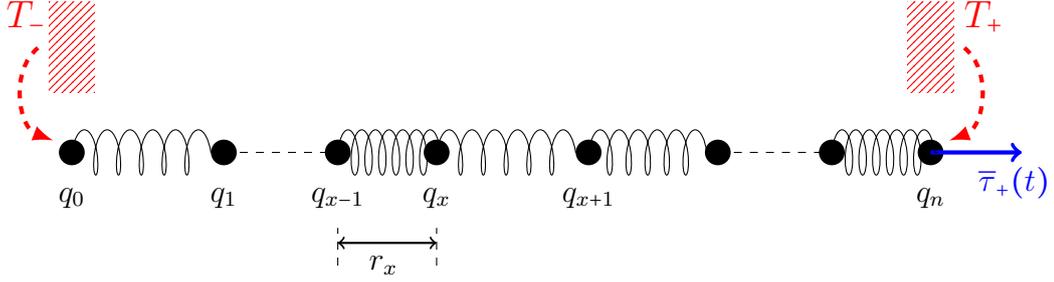

\mm{We let 
\begin{equation}
(\mathbf{ r}_n(t), \mathbf{ p}_n(t)) := (r_1(t), \dots,
r_n(t), p_0(t), \dots, p_n(t)),\quad t\ge0
\label{eq:defrp} \end{equation} be the
$\Omega_n$--valued process whose dynamics is determined by the
equations \eqref{eq:qdynamics}--\eqref{eq:pbd}.
Its generator} is given by
\begin{equation}
 \label{eq:generator}
  L:= n^2 \Big(A + \frac{\gamma}{2} S + \frac{\tilde\gamma}{2} \tilde S \Big),
\end{equation}
where
\begin{align*}
&  A := \sum_{x=1}^n (p_x - p_{x-1}) \partial_{r_x} + \sum_{x=1}^{n-1} (r_{x+1} - r_x) \partial_{p_x}
  + r_1 \partial_{p_0} + \left(\bar\tau_+(t) - r_n\right) \partial_{p_n}\\
&  S f  := \sum_{x=0}^{n-1}  \cX_x \circ \cX_x (f),
\end{align*}
where $\cX_x$ is the {momentum} exchange operator defined as 
\[\cX_x:= p_{x+1}\partial_{p_x} - p_x \partial_{p_{x+1}},\]
and moreover the generator of the Langevin heat baths at the boundaries is given by
\begin{equation*}
  \tilde S :=  T_-\partial_{p_0}^2 - p_0 \partial_{p_0} + T_+\partial_{p_n}^2 - p_n \partial_{p_n}. 
\end{equation*}
\mm{From the microscopic energy conservation law
there exist microscopic energy currents $j_{x,x+1}$ which satisfy
\begin{equation}
  \label{eq:en-evol}
  n^{-2} L \mathcal E_x  = j_{x-1,x} - j_{x,x+1} , \qquad \text{for any } x \in \bar \T_n
\end{equation}
and are given by 
\begin{equation}
\label{eq:current}
  j_{x,x+1} := - p_x r_{x+1} + \frac{\gamma}{2} (p_x^2 - p_{x+1}^2),\qquad \text{if } x \in \{0,...,n-1\}, \end{equation} while at the boundaries 
  \begin{equation} \label{eq:current-bound}
  j_{-1,0} := \frac{\tilde\gamma}{2}  \left(T_- - p_0^2 \right), \qquad
  j_{n,n+1} := -\frac{\tilde\gamma}{2}  \left(T_+ - p_n^2 \right) - \bar\tau_+(t) p_n.
\end{equation}}

\subsection{Macroscopic equations}
\label{sec:macr-equat}
Suppose that $r(t,u)$,  $e(t,u)$, $(t,u)\in  \R_+\times \T
$, are  the macroscopic profiles of \textit{elongation} and \textit{energy}
of the macroscopic system, obtained in the diffusive scaling
limit. The profiles $r(t,\cdot),  \mm{e(t,\cdot)}$ are the expected
limits, as   $n$ gets large, of 
\[ \frac{1}{n}\sum_{x \in \T_n} r_x(t) \delta_{x/n}(\cdot) , \quad \text{ and } \quad \frac{1}{n}\sum_{x \in \bar\T_n} \mm{\mathcal{E}_x(t)} \delta_{x/n}(\cdot), \]
where $\delta_u(\cdot)$ is the delta Dirac function at point $u$. \mm{These convergences are expected to hold in the weak
formulation sense: more details will be given in Appendix \ref{sec:form-deriv-equat}. }
If both convergences do hold at time $t=0$ to some given profiles $r_0(u)$ and $\mathcal{E}_0(u)$, then we expect
that they satisfy the following system of equations\footnote{See also
  Theorems 3.7 and 3.8 of \cite{kos1} for a similar model which gives a similar coupled diffusive system
  for every value of $\gamma$.},
    \begin{align}
        \partial_t
        r(t,u)&={\gamma}^{-1}\;\partial^2_{uu}r(t,u)\vphantom{\Bigg(}\label{eq:linear}
      \\
    \mm{ \partial_t e (t,u)}& \mm{=  \frac 12  \partial^2_{uu}
     \left\{  \big(\gamma^{-1}+\gamma\big) e(t,u) +
       \frac 12 \big(\gamma^{-1}- \gamma\big)\; r^2(t,u)\right\},} \quad
        (t,u)\in \R_+\times \T ,
\label{eq:4}
      \end{align}
with the boundary conditions
\begin{equation*}
  \begin{split}
  & r(t,0) = 0
   , \qquad \qquad r(t,1) = \bar\tau_+(t), \\
   &e (t,0) = T_-, \qquad \quad \ e(t,1) = T_+ + \frac{(\bar\tau_+(t))^2}{2}
  \end{split}
\end{equation*}
and with the initial condition
\begin{equation*}
 r(0,u)= r_0(u),\qquad      
 e(0,u) =\mc E_0(u).
\end{equation*}
In Appendix \ref{sec:form-deriv-equat} we will give the proof arguments for a derivation of these macroscopic equations, which are
    rigorous for $\gamma = 1$, and conditioned to a form of local equilibrium result for $\gamma \neq 1$, stated in \eqref{conj1}.

\medskip

Define now $e_{\rm mech}(t,u):= \frac12 r^2(t,u)$ and $e_{\rm th}(t,u)
:=e(t,u)-e_{\rm mech}(t,u)$ as respectively the {\em mechanical} and {\em thermal}
components of the macroscopic energy. From \eqref{eq:linear} and
\eqref{eq:4} we conclude that 
\[ \partial_t e_{\rm mech}(t,u)= {\ga}^{-1}\left(\partial_{uu}^2
    e_{\rm mech}(t,u)-\big(\partial_ur(t,u)\big)^2\right),\quad
  (t,u)\in \R_+\times \T\]
  with 
 \[
   e_{\rm mech} (t,0) = 0, \quad  e_{\rm mech}(t,1) =
   \frac{(\bar\tau_+(t))^2}{2} ,\quad e_{\rm mech} (0,0) =\frac{ r_0^2(u)}{2},\quad (t,u)\in \R_+\times \T
  \]
and 
\begin{equation}
        \partial_t e_{\rm{th}} (t,u)=
        {\color{red} \frac 12} \big({\gamma}^{-1}+\gamma\big)\partial^2_{uu} e_{\rm{th}}(t,u) +
        {\color{red}{\gamma}^{-1}}\; \big(\partial_u r(t,u)\big)^2, \quad
        (t,u)\in \R_+\times \T ,
\label{eq:linear2}\end{equation} with 
\[ e_{\rm{th}} (t,0) = T_-, \quad  e_{\rm{th}} (t,1) = T_+, \quad t >0.\]

\subsection{Stationary non-equilibrium states}
\label{sec:stat-non-equil}

From now on we assume \mm{$\bar\tau_+(t)\equiv \bar\tau_+$} to be constant in time.

\mm{When $\bar\tau_+=0$ and $T_- = T_+ = T$, the system is \emph{in equilibrium}
  and  the stationary probability distribution is given explicitly by the homogeneous Gibbs measure
  $$
  \nu_{T} (\dd{\bf r},\dd{\bf p}):= g_{T}({\bf r}, {\bf p}) \dd p_0\prod_{x\in\T_n}\dd p_x \dd r_x,
  $$ where
  \begin{equation}
    \label{eq:gibbs}
  g_T({\bf r}, {\bf p}) :=\frac{e^{-p_0^2/2T}}{\sqrt{2\pi T}}
  \prod_{x\in\T_n} \frac{e^{-\mc E_x/T}}{2\pi T}.  
\end{equation}
If $ \bar\tau_+ \neq 0$, or $T_- \neq T_+$, the stationary measure exists and is unique, but it is not given explicitly. More precisely, we know that there exists a unique stationary
probability distribution $\mu_{\rm ss}$  on $\Omega_n$ (cf. \eqref{Omn}) for the microscopic dynamics described by  the equations \eqref{eq:qdynamics}--\eqref{eq:pbd}.
As a consequence
$
\langle LF\rangle_{ {\rm ss}} = 0
$ 
for any function $F$ in the domain of the operator $L$, given by
\eqref{eq:generator}. Hereafter, we denote
$$
\langle F\rangle_{ {\rm ss}} := \int_{\Omega_n} F \; \dd\mu_{\rm ss} .
$$
The proof of the existence and uniqueness of a stationary state follows from the same argument as the one
used in \cite[Appendix A]{bo1} for $\bar\tau_+=0$. The fact that  in our case $\bar \tau_+$ does not vanish requires only minor modifications. In addition, one can show, see bound (A.1) in     \cite{bo1}, that for a fixed $n$ we have $\langle {\cal H}_n\rangle_{ {\rm ss}}<+\infty$ (cf.  \eqref{Hn}).}

The corresponding
\textit{stationary} profiles, denoted respectively by $r_{\rm ss}(u)$ and $e_{\rm th,ss}(u)$,  will solve the stationary version of equations \eqref{eq:linear} and \eqref{eq:linear2}, i.e.:
\begin{equation}
  \label{eq:sspp1}
r_{\rm{ss}} (u) = \bar\tau_+ u
\end{equation} and 
\begin{equation*}  \big(\gamma^{-1}+\gamma\big)\partial^2_{uu} e_{\rm{th, ss}}(u) + 2\gamma^{-1}\; \bar\tau_+^2 = 0, 
\end{equation*}
with the boundary conditions 
\[  
e_{\rm{th, ss}}(0)= T_-, \qquad  e_{\rm{th, ss}}(1) = T_+.
\] 
In other words
\begin{equation}
  \label{eq:3}
   e_{\rm{th, ss}}(u) = \frac{\bar\tau_+^2}{1+\gamma^2}\; u (1-u) + (T_+-T_-)u + T_-.
\end{equation}
Taking the average with respect to the stationary state in
\eqref{eq:en-evol}, we get the {\em stationary microscopic energy current}
\begin{equation}\label{eq:jxs}
 \langle j_{x,x+1}\rangle_{ {\rm ss}} =: \bar j_s, \qquad \text{for any }
 x\in\{-1,\dots,n\}.
\end{equation}
The \textit{macroscopic stationary energy current} is defined as the
limit of  $n\bar
j_s$, as $n\to+\infty$. It equals, see Theorem \ref{theo:current} below,
\begin{equation*}
 J_{\rm ss}=  - \frac12(\gamma^{-1}+\gamma)(T_+ - T_-) -\frac{\bar\tau_+^2}{2\gamma}.
\end{equation*}
Observe that  the energy
current can flow against the temperature gradient if $T_- > T_+$ and $|\bar\tau_+|$ is large enough
(\emph{uphill diffusion}). 
Assuming $T_+ \ge T_-$
  the maximum stationary temperature $e_{\rm{th, ss}}^{\rm max}$ is reached at
$$
u_{\rm max}= \bigg(\frac12+\frac{1+\gamma^2}{\bar\tau_+^2}(T_+-T_-)\bigg) \wedge 1 \,
$$
which implies that, if the condition $2(1+\gamma^2) (T_+ -T_-) \le \bar\tau_+^2$ is satisfied, then
 the maximum temperature of the chain is attained inside, since 
$u_{\rm max} < 1$ (see Figure \ref{fig:temp_profile}), and
 it equals 
\[ e_{\rm{th, ss}}^{\rm max} =  \frac{(T_+ - T_-)}{2} + T_- +\frac {\bar\tau_+^2}{4(1+\gamma^2)}
  \ \ge T_+.
\]
Note that this does not depend on the sign of $\bar\tau_+$. 

\begin{figure}
\centering
\includegraphics[scale=0.35]{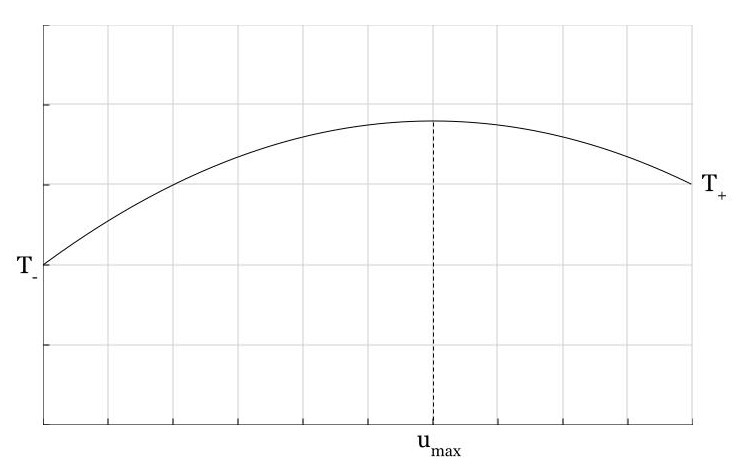}
\caption{Temperature profile when $T_+-T_- < 2 \bar\tau_+^2$.}
\label{fig:temp_profile}
\end{figure}

This phenomenon was observed by dynamical numerical simulations in \cite{iaco} for the
stationary states of the rotor model. It has attracted quite some
interest
from physicists, see \cite{krihsna} for a review.
The present article is devoted to  the proof of
such a phenomenon, when $\gamma=1$.  \mm{This restriction is technical and will be further explained in Section \ref{sec4.3}.}  According to our knowledge it is
the first rigorous proof of this fact in the existing literature.

\section{Main results}
\label{sec:result}

Let us start with the following: 

\begin{theorem}[Stationary elongation profile]\label{theo:elong}
 \mm{The following uniform convergence holds:
\[
 \sup_{u \in \T} \big| \langle r_{[nu]}\rangle_{\rm ss} - r_{\rm ss}(u) \big| \xrightarrow[n\to\infty]{} 0, \]
}
where $r_{\rm ss}(u):=\bar\tau_+ u$. In particular, for any continuous test function $G:\T\to\bb R$, 
\[\frac{1}{n}\sum_{x\in\T_n} G\Big(\frac x n\Big) \langle r_x\rangle_{ {\rm ss}} \xrightarrow[n\to\infty]{} \int_\T G(u) r_{\rm ss}(u) \;\dd u.\]
\end{theorem}
\begin{proof}
The averages under the stationary state $\langle r_x\rangle_{ {\rm ss}}$ and
$\langle p_x\rangle_{ {\rm ss}}$ are computable explicitly, see Proposition
\ref{prop:averages}  in the next section. It turns out that
 $\langle p_x\rangle_{ {\rm ss}}$ is constant for all  $x\in\bar \T_n$ and
 equals  $\bar p_s:= \bar \tau_+/(\gamma n +\tilde\gamma)$ (see \eqref{eq:statvel}).
 From \eqref{eq:statelo} we also have
\begin{equation*}
 n\left(\langle r_{x+1}\rangle_{ {\rm ss}} - \langle r_{x}\rangle_{ {\rm ss}} \right) = n \gamma \bar p_s \
    \xrightarrow[n\to\infty]{} \ \bar \tau_+, \qquad \text{for }  x\in \{1,...,n-1\}
    \end{equation*}
    and \begin{equation*}
\langle r_1\rangle_{ {\rm ss}} \mathop{\longrightarrow}_{n\to\infty} \ 0, \qquad 
    \langle r_n\rangle_{ {\rm ss}} \xrightarrow[n\to\infty]{} \ \bar \tau_+.
\end{equation*}
 \mm{Finally,  \eqref{eq:statelo} directly implies the conclusion of
  the theorem.}
\end{proof}

Concerning the \emph{stationary energy flow} and 
the validity of the Fourier law we show the following result on the
macroscopic stationary energy current.
\begin{theorem}[Stationary energy current and Fourier law]\label{theo:current}
 \begin{equation}  n \bar j_s  \ \xrightarrow[n\to\infty]{}\  - \frac12(\gamma^{-1}+\gamma)(T_+ - T_-) -\frac{\bar\tau_+^2}{2\gamma}.\label{eq:limitcurrent}
\end{equation}
\end{theorem}

Note that Theorems \ref{theo:elong} and \ref{theo:current} are valid
for any $\gamma >0$. We now state our last main result about the
stationary energy profile, which we are able to prove only for
$\gamma=1$. Before stating it we introduce the stationary microscopic
mechanical and thermal  energy   \emph{per particle} as follows
\begin{align*}
&
\mathcal{E}^\mathrm{mech}_x:= \frac12  \langle r_x^2\rangle_{\rm ss}
\\
 &\mathcal{E}^\mathrm{th}_x:= \mathcal{E}_x
  -\mathcal{E}^\mathrm{mech}_x=\frac12 p_x^2 + \frac12 \big(r_x^2 -
  \big\langle r_x^2\rangle_{\rm ss}), \qquad x \in \T_n .
 \end{align*}

%
\begin{theorem}[Stationary energy profile]\label{theo:energy}
Assume that $\gamma=1$. For any continuous test function $G:\T\to\bb
R$ we have 
\begin{align}
\frac{1}{n}\sum_{x\in\T_n} G\Big(\frac x n\Big)
  \mathcal{E}^\mathrm{mech}_x \ &  \xrightarrow[n\to\infty]{} \int_\T
  G(u) \frac12r_{\rm ss}^2(u)\;\dd u,\label{011403-19a}\\
\frac{1}{n}\sum_{x\in\T_n} G\Big(\frac x n\Big) \langle \mc
  E_x ^\mathrm{th}\rangle_{ {\rm ss}} \ & \xrightarrow[n\to\infty]{} \int_\T G(u) e_{\rm
    th,ss}(u) \;\dd u, \label{011403-19b}\\
\frac{1}{n}\sum_{x\in\T_n} G\Big(\frac x n\Big) \langle \mc
  E_x\rangle_{ {\rm ss}}\ &  \xrightarrow[n\to\infty]{} \int_\T G(u) \Big( e_{\rm
    th,ss}(u) + \frac12 r_{\rm ss}^2(u)\Big) \;\dd u, \label{011403-19c}
\end{align}
where 
\begin{align*}
r_{\rm ss}(u)&=\bar\tau_+ u, \\ 
e_{\rm th,ss}(u)&=\frac{\bar\tau_+^2}{2}  u(1-u)+(T_+-T_-)u + T_+.
\end{align*} 
\end{theorem}

The remaining part of the paper deals with the  proofs of Theorems \ref{theo:current} and \ref{theo:energy}.

\section{The stationary state} \label{sec:stat}

Let us start with explicit computations for the average momenta and elongations with respect to the NESS.

\subsection{Elongation and momenta averages}

\begin{proposition}\label{prop:averages}
The average stationary momenta are equal to
\begin{equation}
\langle p_x\rangle_{ {\rm ss}}  = \bar p_s := \frac{\bar \tau_+ }{\gamma n + \tilde\gamma} , \qquad \text{for any } x\in\bar\T_n.  \label{eq:statvel}
\end{equation}
The average stationary elongations are equal to
\begin{equation} \label{eq:statelo}
 \langle r_x\rangle_{ {\rm ss}} =  \frac{\bar p_s}{2} (\tilde\gamma -\gamma + 2 \gamma x) = \frac{\bar\tau_+(2\gamma x+\tilde\gamma-\gamma)}{2(\gamma n + \tilde\gamma)} \qquad \text{for any }  x \in \T_n.
\end{equation}
\end{proposition}
\begin{proof} 
We start with some useful relations that hold for the stationary state:
\begin{itemize}
\item[(1)] since $\langle Lr_x\rangle_{ {\rm ss}} = 0$,
applying \eqref{eq:generator}, we conclude
$$
\langle p_x\rangle_{ {\rm ss}} = \langle p_{x-1}\rangle_{ {\rm ss}} = \bar p_s,\qquad \text{for any } x \in \T_n;
$$
\item[(2)] from $\langle Lp_x\rangle_{ {\rm ss}} = 0$ we get
\begin{align*}
   &\langle r_{x+1}\rangle_{ {\rm ss}} - \langle r_{x}\rangle_{ {\rm ss}} = \gamma \bar p_s, \qquad \text{for any } x\in\{2, \dots, n-2\} \\
    & \langle r_{1}\rangle_{ {\rm ss}} = \frac12(\gamma+{ \tilde\gamma}) \bar p_s,\\
    &\langle r_{n}\rangle_{ {\rm ss}}  = - \frac12(\gamma+{ \tilde\gamma}) \bar p_s + \bar\tau_+.
  \end{align*}

\end{itemize}
These equations determine the average stationary momentum and
elongation as given in formulas \eqref{eq:statvel}
  and \eqref{eq:statelo}.
\end{proof}

\subsection{Elements of the proofs of Theorems \ref{theo:current}
  and \ref{theo:energy}}
 \label{sec4.3}
 
 One of the main characteristics of this model is the existence of
 an explicit \emph{fluctuation-dissipation relation},
 which permits to write the stationary current $\bar j_s$
 as a discrete gradient of some local function, as given in the following:

\begin{proposition}[Decomposition of the stationary current] \label{prop:current}
We can write $\bar j_s$ as a discrete gradient, namely 
\begin{equation}\label{eq:gradientj}
\bar j_s=\nabla \phi(x):=\phi(x+1)-\phi(x),\qquad x\in
\{1,\dots,n-1\} ,
\end{equation}
with 
\begin{equation}
\phi(x): = -\frac{1}{2\gamma}\left( \langle r_{x}^2\rangle_{ {\rm ss}} +
  \langle p_{x-1} p_{x}\rangle_{ {\rm ss}}  \right) - \frac{\gamma}{4}\big(
\langle p_x^2 \rangle_{ {\rm ss}} + \langle p_{x-1}^2
\rangle_{ {\rm ss}}\big) , \qquad x\in \T_n .  \label{eq:phi1}
\end{equation}
\end{proposition}

\begin{remark}
Thanks to \eqref{eq:gradientj}, the function $\phi(x)$ is  
\emph{harmonic}, i.e.: 
\begin{equation} \label{eq:phiharmonic}
 \Delta \phi(x):=\phi(x+1)+\phi(x-1)-2\phi(x)=0,  \qquad \text{for any } x \in \{2,\dots,n-1\}.
 \end{equation} 
\end{remark}

\begin{proof}[Proof of Proposition \ref{prop:current}]
By a direct calculation one can easily check that the energy currents $j_{x,x+1}$ (defined in \eqref{eq:current}) satisfy the following \emph{fluctuation-dissipation relation}: 
\begin{equation}
\label{eq:fdst2}
j_{x,x+1} = n^{-2} Lg_x -\frac{1}{2\gamma}\nabla \big(r_x^2 + p_{x-1}p_x\big) - \frac{\gamma}{4} \nabla \big(p_{x-1}^2 + p_x^2\big),
\end{equation}  
  for any $x \in \{1,...,n-1\}$,  with 
  \[
  g_x:=-\frac{1}{4}p_x^2 + \frac{1}{2\gamma}p_x(r_x+r_{x+1}).
  \]
Therefore,  \eqref{eq:gradientj} is obtained by taking the average in \eqref{eq:fdst2} with respect to the stationary state.
\end{proof}

We can now sketch the proof of Theorem \ref{theo:energy}: straightforward computations, using the definition \eqref{eq:phi1} of $\phi$, yield
\begin{align}
\langle {\cal E}_x\rangle_{ {\rm ss}}= \frac{1}{2}\big(\langle p_x^2 \rangle_{ {\rm ss}} + \langle r_x^2 \rangle_{ {\rm ss}} \big) & = -\frac{2\gamma}{1+\gamma^2} \phi(x) + \frac{\gamma^2}{2(1+\gamma^2)} \big(\langle p_x^2 \rangle_{ {\rm ss}} - \langle p_{x-1}^2 \rangle_{ {\rm ss}} \big) \notag\\
& \quad - \frac{1}{1+\gamma^2} \langle p_x p_{x-1} \rangle_{ {\rm ss}}+ \frac{1-\gamma^2}{2(1+\gamma^2)} \big(\langle p_x^2 \rangle_{ {\rm ss}} - \langle r_x^2 \rangle_{ {\rm ss}} \big).  \label{eq:decomp}
\end{align}
Therefore, the microscopic energy profile can be decomposed as the sum of four terms: 
\begin{equation}
\label{energy:decomp}
{\cal H}_n(G):= \frac{1}{n} \sum_{x \in \T_n}
  G\Big(\frac{x}{n}\Big) \langle \mathcal{E}_x\rangle_{ {\rm ss}} =  \mc
  H_n^{\phi}(G) +  \mc H_n^{\nabla}(G) +  \mc H_n^{\text{corr}}(G) +
  \mc H_n^{\text{m}}(G),
\end{equation}
where 
\begin{align*}
\mc H^\phi_n(G) & := -\frac{2\gamma}{1+\gamma^2}\; \frac1n  \sum_{x \in \T_n} G\Big(\frac{x}{n}\Big) \phi(x), \\
\mc H^\nabla_n(G) & := \frac{\gamma^2}{2(1+\gamma^2)} \; \frac1n\sum_{x\in\T_n} G \Big(\frac x n \Big) \big( \langle p_x^2 \rangle_{ {\rm ss}} - \langle p_{x-1}^2 \rangle_{ {\rm ss}} \big),\\
\mc H^{\text{corr}}_n(G) & :=-\frac{1}{1+\gamma^2} \; \frac1n \sum_{x \in \T_n} G\Big(\frac{x}{n}\Big) \langle p_{x-1}p_x\rangle_{ {\rm ss}}, \\
\mc H^{\text{m}}_n(G) & := \frac{1-\gamma^2}{2(1+\gamma^2)} \; \frac{1}{n} \sum_{x \in \T_n} G\Big(\frac{x}{n}\Big) \big(\langle p_x^2 \rangle_{ {\rm ss}} -\langle  r_x^2\rangle_{ {\rm ss}}\big) .
\end{align*}
Note that, if $\gamma=1$, then $\mathcal{H}_n^{\text{m}}\equiv 0.$  \mm{If $\gamma \neq 1$, we conjecture that this last term vanishes as $n\to\infty$,
  but we are not able to prove it at the moment.}
The limits of the other three terms will be obtained in the next section
and are summarized in the following proposition: 
\begin{proposition} \label{prop:limitsE}
For any continuous test function $G:\T\to\R$, 
\begin{align}
\mathcal{H}_n^\phi(G) & \xrightarrow[n\to\infty]{} \int_\T G(u) \Big( \frac{\bar\tau_+^2}{1+\gamma^2}u + (T_+-T_-)u+T_- \Big)\; \dd u, \label{eq:limitEphi} \\
\mathcal{H}_n^\nabla(G) & \xrightarrow[n\to\infty]{} 0 ,\label{eq:limitEnabla}\\ 
\mathcal{H}_n^{\mathrm{corr}}(G) & \xrightarrow[n\to\infty]{} 0. \label{eq:limitEcorr}
\end{align}
\end{proposition}
The complete proof of the proposition will be given in Section \ref{sec:proofprop}. Let us first comment
on the ideas used in the argument. The limit \eqref{eq:limitEphi}
will be concluded using the fact that $\phi$ is
harmonic,   \eqref{eq:limitEnabla} is a consequence of the presence of a discrete gradient $\langle p_x^2\rangle_{ {\rm ss}} - \langle p_{x-1}^2 \rangle_{ {\rm ss}}$ inside the sum, and \eqref{eq:limitEcorr} will be  shown thanks to the second order bounds, which are obtained in the next section. 

\begin{proof}[Proof of Theorem \ref{theo:energy}]
With the help of Proposition \ref{prop:limitsE}
the proof of Theorem
\ref{theo:energy} becomes straightforward. Assume that $\gamma=1$. From the decomposition \eqref{eq:decomp} and Proposition \ref{prop:limitsE}, we get
\begin{align*}\frac{1}{n}\sum_{x\in\T_n} G\Big(\frac x n \Big) \langle \mc E_x \rangle_{ {\rm ss}} \xrightarrow[n\to\infty]{} &\int_\T G(u) \Big(  \frac{\bar\tau_+^2}{2} \; u +  (T_+-T_-)u+T_- \Big)\; \dd u \\
&=\int_\T G(u) \Big( \frac{\bar\tau_+^2}{2} u(1-u)
  +(T_+-T_-)u+T_-+\frac{(\bar\tau_+ u)^2}{2} \Big) \;\dd u.
\end{align*}
Recalling \eqref{eq:sspp1} and \eqref{eq:3} we conclude that the right
hand side equals
$$
 \int_\T G(u) \Big( e_{\rm th,ss}(u)+\frac{r_{\rm ss}^2(u)}{2} \Big)
 \;\dd u.
$$
 \mm{Thus   \eqref{011403-19c} follows. From Theorem
\ref{theo:elong} we immediately conclude 
\eqref{011403-19a}. The convergence of the stationary microscopic
thermal energy profile in \eqref{011403-19b} is an immediate consequence of these two statements. }
\end{proof}

\section{Moment bounds under the stationary state}
\label{sec:moment}
In this section we present a complete proof of Proposition
  \ref{prop:limitsE} (see Section \ref{sec:proofprop}) and we show Theorem \ref{theo:current}
  (see Proposition \ref{prop:limit-current}).
Before presenting the
proof, we need a few technical estimates on the entropy production (Section \ref{sec:entr-prod-stat}) and second order moments (Section \ref{sec:bounds-second-moment} and Section \ref{sec:bounds-energy}). In the whole section we do not assume $\gamma=1$, since our results hold for any $\gamma >0$.
 
\subsection{Entropy production of the stationary state}
\label{sec:entr-prod-stat}

 Recall definition \eqref{eq:gibbs}. For a given $T$ we will use
  \[\nu_{T} (\dd{\bf r},\dd{\bf p})= g_{T}({\bf r}, {\bf p}) \dd p_0\prod_{x\in\T_n}\dd p_x \dd r_x ,\]
  as a reference measure, 
  and denote its respective expectation by $\llangle \cdot \rrangle_{T}$.

Stationarity of $\mu_{\rm ss}$  under the microscopic dynamics implies that $L^\star \mu_{\rm ss}=0$ 
(in the sense of distributions).  The operator $L^\star$ is hypoelliptic, thus by   \cite[Theorem 1.1, p.~149]{hormander}, the
measure $\mu_{\rm ss}$ has a  smooth density $f_s$  with respect  to  $\nu_{T_+}$, i.e.
\begin{equation*}
  \left\langle F \right\rangle_{ {\rm ss}} = \llangle F f_s \rrangle_{T_+} \; = \int F f_s \; \dd\nu_{T_+} .
\end{equation*}

\begin{proposition}[Entropy production] \label{prop:entropy}
Denote $h:=g_{T_-}/g_{T_+}$.The following formula holds 
\begin{multline}
  \label{eq:8}
  \gamma \sum_{x=0}^{n-1} {\mc D}_x (f_s) + 
 \tilde\gamma T_- \; \Big\llangle   \frac{(\partial_{p_0} (f_s/h))^2}{(f_s/h)} \Big\rrangle_{T_-}+ 
\tilde\gamma T_+ \; \Big\llangle  \frac{(\partial_{p_n} f_s)^2}{f_s} \Big\rrangle_{T_+}\\  = 
  \frac{\bar \tau_+^2  }{T_+(\gamma n+\tilde\gamma)} + \tilde\gamma \left(\frac{1}{T_+}- \frac{1}{T_-}\right) \big( T_- - \langle p_0^2\rangle_{ {\rm ss}} \big),
\end{multline}
where 
\[{\mc D}_x (f_s):=\Big\llangle
  \frac{\left(\cX_{x}f_s\right)^2}{f_s} \Big\rrangle_{T_+}.\]
\end{proposition}

\begin{proof}
Integration by parts yields
$$
\left\llangle  g\tilde Sf
\right\rrangle_{T_+}=T_+\left\llangle \partial_{p_n}f \partial_{p_n}g
\right\rrangle_{T_+} +T_-\left\llangle \partial_{p_0}f \partial_{p_0}g
\right\rrangle_{T_+}
+T_-\left(\frac{1}{T_-}-\frac{1}{T_+}\right)\left\llangle g \partial_{p_0}f  \right\rrangle_{T_+} 
$$
for any $f,g\in C^\infty_0(\Om_n)$, where  $C^\infty_0(\Om_n)$ is the
space of compactly supported smooth functions. As 
$$
A\left(\mathcal{H}_n (\mathbf r, \mathbf p)\right)=\bar\tau_+p_n \quad \mbox{and}\quad {\cX}_x\Big(\sum_{y=0}^np_{y}^2\Big)=0,\qquad x\in \{0,\ldots,n-1\}
$$
we conclude
\begin{align*}
&
 -\left\llangle  gAf  \right\rrangle_{T_+} =\left\llangle   f Ag \right\rrangle_{T_+}+\bar\tau_+\left\llangle p_n f g \right\rrangle_{T_+} \vphantom{\Big(}
,\\
&
 -\left\llangle  g {\cX}_x^2f \right\rrangle_{T_+} =\left\llangle   {\cX}_xf {\cX}_xg \right\rrangle_{T_+} ,  \vphantom{\Big(}
\end{align*}
for any $x=0,\ldots,n-1,$ and any $f,g\in C^\infty_0(\Om_n)$.
 We take the average of  $-n^{-2}L (\log f_s) $ with respect
  to the stationary state $\mu_{\rm ss}$. Taking into account the  above 
  identities  we obtain
\begin{align*}
    0 &=  - n^{-2} \left\langle  L  \log f_s \right\rangle_{ {\rm
        ss}} = -n^{-2} \left\llangle  f_s  L  \log f_s
        \right\rrangle_{T_+}  \vphantom{\bigg(} \notag \\
    &=\gamma \sum_{x=0}^{n-1} \mc D_x(f_s) + 
    \tilde\gamma T_+\; \Big\llangle   \frac{(\partial_{p_n} f_s)^2}{f_s} \Big\rrangle_{T_+} 
    - \bar\tau_+ \left\llangle  \partial_{p_n} f_s \right\rrangle_{T_+} - \tilde\gamma \left\langle \big(T_- \partial_{p_0}^2 - p_0\partial_{p_0}\big) \log f_s \right\rangle_{ {\rm ss}}. 
\end{align*}
From the definition $h=g_{T_-}/g_{T_+}$,  the last term can be rewritten in
the form: 
\begin{align*}
- \left\langle \big(T_- \partial_{p_0}^2 - p_0\partial_{p_0}\big) \log f_s \right\rangle_{ {\rm ss}} & = - \int \frac{f_s}{h} \big(T_- \partial_{p_0}^2 - p_0\partial_{p_0}\big) \Big( \log \Big(\frac{f_s}{h}\Big)\Big) \; g_{T_-} dp_0\prod_{x=1}^ndp_x dr_x \\ & \quad - \int f_s   \big(T_- \partial_{p_0}^2 - p_0\partial_{p_0}\big)\big(\log h\big) \; g_{T_+} dp_0\prod_{x=1}^ndp_x dr_x \\
& = T_-\;  \Big\llangle   \frac{(\partial_{p_0} (f_s/h))^2}{(f_s/h)} \Big\rrangle_{T_-} + \left(\frac{1}{T_-}- \frac{1}{T_+}\right) \big( T_- - \langle p_0^2\rangle_{ {\rm ss}} \big). 
\end{align*}
Moreover, by integration by parts and \eqref{eq:statvel}, we obtain
\begin{equation*}
  \left\llangle  \partial_{p_n} f_s \right\rrangle_{T_+} = T_+^{-1} \left\llangle  p_n f_s \right\rrangle_{T_+} = \frac{\bar p_s}{T_+}  
= \frac{\bar \tau_+ }{T_+(\gamma n + \tilde\gamma)}
\end{equation*}
and \eqref{eq:8} follows. Since $f_s$ needs not be compactly supported
the above calculation is somewhat formal. A rigorous argument (using
variational principles) can be found in \cite[Section 3]{bo2}.
\end{proof}

\subsection{Bounds on second moments}
\label{sec:bounds-second-moment}

In the present section we obtain some bounds on the covariance
functions of momenta and positions,  with respect to the stationary
states. In particular, we estimate the magnitude of
the average current $|\bar j_s|$, see
\eqref{eq:jxs}  and investigate the behaviour of $\phi(1)$
and $\phi(n)$ as $n\to\infty$, see \eqref{eq:phi1}. 

Let us first state a rough estimate on the second moments at the boundaries, which is going to be refined further.

\begin{proposition}[Second moments  {at} the boundaries:
 Part I]\label{prop:second-bnd}
The following equality holds
\begin{equation}
  \label{eq:10}
 \langle p_0^2 \rangle_{ {\rm ss}} + \langle p_n^2 \rangle_{ {\rm
     ss}} = T_+ + T_- + \frac{2\bar \tau_+ \bar
   p_s}{\tilde\gamma}\qquad\mbox{for all }n\ge 1.
\end{equation}
Moreover, there exists a constant
$C=C(\gamma,\tilde\gamma,\bar\tau_+,T_+,T_-)>0$, such that
\begin{equation}
  \label{eq:23}
   \langle r_1^2\rangle_{ {\rm ss}}  +  \langle
   r_n^2\rangle_{ {\rm ss}} \le C\qquad \mbox{for all }n\ge 1.
\end{equation}
\end{proposition}
\begin{remark}
{\em {By convention},  the constants appearing in the statements below depend only on the
parameters indicated in parentheses in the statement of the proposition.}
\end{remark}
\begin{proof}[Proof of Proposition \ref{prop:second-bnd}] 
The first identity \eqref{eq:10} is an easy consequence of
\eqref{eq:current-bound} and \eqref{eq:jxs}, which  yields
\begin{align}
\bar j_s & =   \frac{\tilde\gamma}{2}  \left(T_- - \langle p_0^2\rangle_{ {\rm ss}} \right),\label{eq:momentrel2}\\
     \bar j_s &= - \bar\tau_+ \bar p_s- \frac{\tilde\gamma}{2}  \left(T_+- \langle p_n^2\rangle_{ {\rm ss}} \right). \label{eq:momentrel3}
  \end{align}
 Identity \eqref{eq:10} is obtained by adding sideways the above
equalities.  
To show estimate \eqref{eq:23} note that
\begin{align}
    n^{-2}L(p_0 r_1) &= (p_1-p_0)p_0+  r_1^2 - \frac12(\tilde\gamma +\gamma)
    p_0r_1 \label{eq:a}\\
     n^{-2} L(p_n r_n) &= p_n(p_n - p_{n-1}) + (\bar \tau_+-r_n) r_n - \frac12(\tilde\gamma +\gamma)
    p_n r_n .\label{eq:b}
  \end{align}
After taking the
average with respect to the stationary state from \eqref{eq:a} we conclude 
\begin{align}
\label{022109}
    \langle r_1^2\rangle_{ {\rm ss}} &= \langle p_0^2\rangle_{ {\rm ss}}-\langle p_1p_0
                             \rangle_{ {\rm ss}}+\frac12(\tilde\gamma +\gamma)
                        \langle p_0r_1\rangle_{ {\rm ss}}.
 \end{align}   
Recalling the definition of the current \eqref{eq:current} and
then invoking \eqref{eq:momentrel2}, we get
\begin{align*}
   \langle r_1^2\rangle_{ {\rm ss}}   &=\langle p_0^2\rangle_{ {\rm ss}}-\langle p_1p_0 \rangle_{ {\rm ss}}-\frac12(\tilde\gamma +
    \gamma)\Big(\langle j_{0,1}\rangle_{ {\rm ss}} + \frac{\gamma}{2} \big(\langle p_1^2 \rangle_{ {\rm ss}} - \langle p_0^2 \rangle_{ {\rm ss}}\big) \Big) \\ 
    & =\langle p_0^2\rangle_{ {\rm ss}}-\langle p_1p_0 \rangle_{ {\rm ss}}- \frac{\tilde \gamma}{4} (\tilde\gamma +
    \gamma)\big( {T_-}-\langle p_0^2 \rangle_{ {\rm ss}}\big)-\frac{\gamma}{4}(\tilde\gamma+\gamma) \big(\langle p_1^2 \rangle_{ {\rm ss}} - \langle p_0^2 \rangle_{ {\rm ss}}\big)  .   
 \notag
\end{align*}
Using Young's inequality 
\[ \big| \langle p_1p_0 \rangle_{ {\rm ss}} \big| \le \frac{A}{2} \langle p_1^2 \rangle_{ {\rm ss}} + \frac{1}{2A} \langle p_0^2 \rangle_{ {\rm ss}}, \]
with $A = \frac\gamma 2 (\gamma + \tilde \gamma)$, 
we get 
  \begin{equation}
  \langle r_1^2 \rangle_{ {\rm ss}} \le \Big( \frac{1}{\gamma(\gamma+\tilde \gamma)} + 1 + \frac14(\gamma+\tilde\gamma)^2 \Big) \langle p_0^2 \rangle_{ {\rm ss}}.\label{eq:r1-2}
\end{equation}
From \eqref{eq:10} we conclude that $\langle r_1^2 \rangle_{ {\rm ss}}$ is bounded.

To estimate $\langle r_n^2 \rangle_{ {\rm ss}}$, note that from \eqref{eq:b} we write
\begin{equation} \label{eq:rn} \langle r_n^2\rangle_{ {\rm ss}} = \langle p_n^2\rangle_{ {\rm ss}} -  \langle p_n p_{n-1}\rangle_{ {\rm ss}} + \bar \tau_+ \langle r_n\rangle_{ {\rm ss}} -  \frac12(\tilde\gamma +\gamma) \langle p_nr_n\rangle_{ {\rm ss}}.\end{equation}
We use again Young's inequality 
\[
\big| \langle p_np_{n-1} \rangle_{ {\rm ss}} \big| \le \frac{A}{2} \langle p_n^2
\rangle_{ {\rm ss}} + \frac{1}{2A} \langle p_{n-1}^2 \rangle_{ {\rm ss}} ,
\] 
with $A= 1/(2\gamma)$ and we get
\begin{equation} \label{eq:pop}\langle r_n^2 \rangle_{ {\rm ss}} \le   \Big(1+\frac{1}{4\gamma}\Big) \langle p_n^2 \rangle_{ {\rm ss}} +  \gamma \langle p_{n-1}^2 \rangle_{ {\rm ss}} + \tau \langle r_n \rangle_{ {\rm ss}} - \frac12(\tilde\gamma +  \gamma) \langle p_n r_n\rangle_{ {\rm ss}}. \end{equation}
To replace $\langle p_{n-1}^2 \rangle_{ {\rm ss}}$, note that 
\begin{equation} \label{eq:p2-p2}
n^{-2}L(p_n^2) = 2 (\bar\tau_+ - r_n) p_n + \gamma (p_{n-1}^2 - p_n^2)
+  \tilde\gamma (T_+-p_n^2). 
\end{equation}
Taking the average with respect to the stationary state, we obtain:
\begin{equation}\label{eq:pn-1} \gamma\langle p_{n-1}^2 \rangle_{ {\rm ss}} =2 \langle r_np_n\rangle_{ {\rm ss}} - 2 \bar\tau_+ \langle p_n \rangle_{ {\rm ss}} +  (\gamma+\tilde\gamma) \langle p_n^2 \rangle_{ {\rm ss}} -  \tilde \gamma T_+,  \end{equation}
which, in \eqref{eq:pop}, gives
\[\langle r_n^2 \rangle_{ {\rm ss}} \le   \Big(1+\frac{1}{4\gamma}+\gamma + \tilde\gamma\Big) \langle p_n^2 \rangle_{ {\rm ss}}  +\bar\tau_+\; \big(\langle r_n \rangle_{ {\rm ss}} -2\langle p_n\rangle_{ {\rm ss}}\big)+ \frac12(4-\tilde\gamma - \gamma) \langle p_n r_n\rangle_{ {\rm ss}} -  \tilde\gamma T_+.\]
  Using again Young's inequality 
  \[\big| \langle p_nr_{n} \rangle_{ {\rm ss}} \big| \le \frac{A}{2} \langle
    p_n^2 \rangle_{ {\rm ss}} + \frac{1}{2A} \langle r_{n}^2 \rangle_{ {\rm ss}} \] with
  $A=\frac12|4-\gamma-\tilde\gamma|$, we finally arrive at 
   \begin{equation}\label{eq:rn-2}
  \frac12 \langle r_n^2 \rangle_{ {\rm ss}} \le   \Big(1+\frac{1}{4\gamma}+\gamma + \tilde\gamma+\frac14(4-\gamma-\tilde\gamma)^2\Big) \langle p_n^2 \rangle_{ {\rm ss}}  +\bar\tau_+\; \big(\langle r_n \rangle_{ {\rm ss}} -2\langle p_n\rangle_{ {\rm ss}}\big) -  \tilde\gamma T_+.
\end{equation}
We now invoke \eqref{eq:10}, \eqref{eq:statvel} and \eqref{eq:statelo}
to conclude the bound on $\langle r_n^2\rangle_{ {\rm ss}}$, which combined with the
already obtained bound on $\langle r_1^2 \rangle_{ {\rm ss}}$ yields  \eqref{eq:23}.
\end{proof}

\begin{corollary}[Second moments at the boundaries: Part II]\label{cor:bound}
There exists $C=C(\gamma,\tilde\gamma,\bar\tau_+,T_+,T_-)>0$, such that 
\begin{equation}
\label{p1andpn-1} 
\langle p_1^2 \rangle_{ {\rm ss}} + \langle p_{n-1}^2 \rangle_{ {\rm ss}} \le C\qquad \mbox{for all }n\ge 1.
\end{equation}
\end{corollary}

\begin{proof}

To bound  $\langle p_{n-1}^2 \rangle_{ {\rm ss}}$ we use formula \eqref{eq:pn-1}.
 From an elementary inequality $|\langle r_n p_n
 \rangle_{ {\rm ss}} |\le \frac12(\langle p_n^2 \rangle_{ {\rm ss}} + \langle r_n^2 \rangle_{ {\rm ss}})$
 and Proposition \ref{prop:second-bnd}, we easily conclude that
 $\langle p_{n-1}^2\rangle_{ {\rm ss}}$ is bounded.

The  bound for $\langle
 p_1^2 \rangle_{ {\rm ss}}$ is obtained similarly. First, note that 
\begin{equation}\label{eq:p1} n^{-2} L(p_0^2) = 2 r_1 p_0 +  \gamma
  (p_1^2-p_0^2) + \tilde\gamma (T_--p_0^2).
\end{equation}
 Taking the average with respect to the stationary state, using the
 inequality $|\langle r_1 p_0 \rangle_{ {\rm ss}} |\le \frac12(\langle r_1^2 \rangle_{ {\rm ss}} +
 \langle p_0^2\rangle_{ {\rm ss}})$, and invoking Proposition
 \ref{prop:second-bnd}, we conclude the desired bound on $\langle
 p_1^2 \rangle_{ {\rm ss}}$. Thus \eqref{p1andpn-1} follows.
\end{proof}

In the next proposition we provide a bound on the energy current
under the stationary state, which will be   further refined in Proposition \ref{prop:limit-current}.

\begin{proposition}[The stationary current: Part I]
\label{prop:current-stationary}
There exists a constant $C=C(\gamma,\tilde\gamma,\bar\tau_+,T_+,T_-)
>0$, such that the stationary current satisfies
\begin{equation}  \big|\bar j_s\big| \leqslant\frac C n\qquad \mbox{for all }n\ge 1.\label{eq:vanishj}\end{equation}
\end{proposition}

\begin{proof}
We sum the identity \eqref{eq:gradientj} from $x=1$ to $n-1$ and apply
\eqref{eq:phi1} to express $\phi(n)$ and $\phi(1)$.
In this way we get 
\begin{align} (n-1) \bar j_s & = \phi(n)-\phi(1) \notag \\
& = \Big\langle - \frac{p_{n-1}p_n + r_n^2}{2\gamma} - \frac{\gamma (p_{n-1}^2 +p_n^2)}{4} \Big\rangle_{ {\rm ss}} +  \Big\langle  \frac{p_{1}p_0 + r_1^2}{2\gamma} + \frac{\gamma (p_{1}^2 +p_0^2)}{4} \Big\rangle_{ {\rm ss}}.  \label{eq:njs}
\end{align}
Therefore, \eqref{eq:vanishj} follows from the elementary 
  inequalities
 $$
|\langle p_{x-1}p_x\rangle_{ {\rm ss}} |\le \frac12 \big( \langle p_x^2
\rangle_{ {\rm ss}} + \langle p_{x-1}^2 \rangle_{ {\rm ss}} \big)\quad
\mbox{ for $x=n$ and $x=1$}
$$
together with  the bounds obtained in Proposition \ref{prop:second-bnd} and Corollary \ref{cor:bound}. 
\end{proof}

Proposition \ref{prop:current-stationary} permits to get a better estimate
on the entropy production. Namely, combining  \eqref{eq:8},
\eqref{eq:momentrel2} and \eqref{eq:vanishj} we conclude the following.
\begin{corollary}\label{cor:betterbound}
There exists $C=C(\gamma,\tilde\gamma,\bar\tau_+,T_+,T_-)>0$, such that 
\begin{equation}
 \gamma \sum_{x=0}^{n-1} {\mc D}_x (f_s) + 
 \tilde\gamma T_- \; \Big\llangle   \frac{(\partial_{ p_0} (f_s/h))^2}{(f_s/h)} \Big\rrangle_{T_-} + 
\tilde\gamma T_+ \; \Big\llangle  \frac{(\partial_{p_n} f_s)^2}{f_s}
\Big\rrangle_{T_+} \leq \frac{C}{n},\qquad n\ge1. \label{eq:88}
\end{equation}
\end{corollary}

Thanks to Proposition \ref{prop:current-stationary} we are now able to
estimate the covariances of  momenta and stretches at the  boundaries as follows: 

\begin{proposition}[Second moment at the boundaries: Part III] \label{prop:rnpn}
There exists $C=C(\gamma,\tilde\gamma,\bar\tau_+,T_+,T_-)>0$, such
that, at the left boundary point
\begin{equation}  \big|\langle p_0p_1\rangle_{ {\rm ss}}\big|+\big|\langle
  r_1p_1\rangle_{ {\rm ss}}\big|+\big|\langle r_1p_0\rangle_{ {\rm ss}}\big| \leqslant
  \frac{C}{\sqrt n}, \qquad n\ge1, \label{eq:vanishbd1}\end{equation} and at the
right boundary point
  \begin{equation}
\big|\langle p_np_{n-1}\rangle_{ {\rm ss}}\big| +\big|\langle
r_np_{n}\rangle_{ {\rm ss}}\big|+\big|\langle r_np_{n-1}\rangle_{ {\rm ss}}\big| \leqslant
\frac{C}{\sqrt n},\qquad n\ge1. \label{eq:vanishbd2}
\end{equation}
\end{proposition}
\begin{proof} 
Integration by parts yields
$$
\langle p_0 p_1\rangle_{ {\rm ss}}=-T_-\big\llangle p_1 (f_s/g_{T-}) \partial_{p_0}g_{T-}\big\rrangle_{T_+}=T_-\big\llangle p_1 \partial_{p_0} (f_s/h)\big\rrangle_{T_-}.
$$
We use the entropy production bound \eqref{eq:88} and estimate
\eqref{p1andpn-1} on $\langle p_1^2 \rangle_{ {\rm ss}}$, to estimate the right
hand side. As a result we get
\begin{equation*}
   \big| \langle p_0 p_1\rangle_{ {\rm ss}} \big|= T_-\big|\big\llangle p_1 \partial_{p_0} (f_s/h)\big\rrangle_{T_-}\big| \le T_-\big\langle p_1^2\big\rangle_{ {\rm ss}}^{\frac12} \; 
    \Big\llangle  \frac{( \partial_{p_0} (f_s/h))^2}{(f_s/h)}\Big\rrangle_{T_-}^{\frac12} 
    \le \frac{C}{\sqrt n}.
\end{equation*}
Similarly, 
\begin{equation*}
   \big| \langle p_n p_{n-1}\rangle_{ {\rm ss}} \big|=T_+ \big|\big\llangle p_{n-1} \partial_{p_n} f_s\big\rrangle_{T_+}\big| \le T_+\big\langle p_{n-1}^2\big\rangle_{ {\rm ss}}^{\frac12} \; 
    \Big\llangle  \frac{( \partial_{p_n} f_s)^2}{f_s}\Big\rrangle_{T_+}^{\frac12} 
    \le \frac{C}{\sqrt n}.
\end{equation*}
Finally, note that, for any $x \in \T_n$ 
\begin{equation}
\label{013108-19} n^{-2}L(r_x^2) = 2(p_x-p_{x-1})r_x.\end{equation}
 Therefore,  upon averaging with respect to the NESS, we get
\begin{equation}
\label{eq:equ}
\langle p_x r_x \rangle_{ {\rm ss}} = \langle p_{x-1} r_x \rangle_{ {\rm ss}},\quad  x \in \T_n.
\end{equation}
In particular, applying \eqref{eq:equ} for $x=1$ and $x=n$, we remark
that    the only quantities we need to yet estimate
  are  $| \langle r_1 p_0 \rangle_{ {\rm ss}} | $ and $| \langle r_np_n
\rangle_{ {\rm ss}}|$. This is done  using the entropy production bound
\eqref{eq:88} in the same manner as before, namely: 
\[
   \big| \langle r_1 p_0\rangle_{ {\rm ss}} \big|= T_-\big|\big\llangle r_1 \partial_{p_0}( f_s/h)\big\rrangle_{T_-}\big| \le T_-\big\langle r_1^2\big\rangle_{ {\rm ss}}^{\frac12} \; 
    \Big\llangle  \frac{( \partial_{p_0} (f_s/h))^2}{(f_s/h)}\Big\rrangle_{T_-}^{\frac12} 
    \le \frac{C}{\sqrt n},
\]
from \eqref{eq:23} and \eqref{eq:8}. We leave the last estimate for the reader. 
\end{proof}

We now have all the ingredients necessary to prove moments convergences at the boundaries:
\begin{corollary}[Second moments at the boundaries: Part IV]
\label{prop:limit-stationary}
The following limits hold: at the left boundary point,  
\begin{align}
  \langle p_x^2\rangle_{ {\rm ss}} \ &\xrightarrow[n\to\infty]{}\ T_-\qquad  \text{ for } x\in\{0,1\}, \label{eq:limitp}\\  
     \langle r_1^2\rangle_{ {\rm ss}}  \ &\xrightarrow[n\to\infty]{}\ T_- ,\label{eq:limit12}\\
       \langle r_1r_2\rangle _s \ &\xrightarrow[n\to\infty]{}0,  \label{eq:limit13}\end{align}
and at the right boundary point,
       \begin{align}
    \langle p_x^2\rangle_{ {\rm ss}} \ &\xrightarrow[n\to\infty]{}\ T_+\qquad  \text{ for } x\in\{n-1,n\} ,\label{eq:limitp2}\\ 
  \langle r_n^2\rangle_{ {\rm ss}} \ &\xrightarrow[n\to\infty]{}\ T_+ + \bar\tau_+^2, \label{eq:limit11}\\
  \langle r_{n-1} r_n\rangle _s \ &\xrightarrow[n\to\infty]{}\  \bar\tau_+^2.\label{eq:limit14} \end{align}
\end{corollary}

\begin{proof}[Proofs of  (\ref{eq:limitp}) and (\ref{eq:limitp2})]

From \eqref{eq:momentrel2} and Proposition
\ref{prop:current-stationary} we get $\langle p_0^2 \rangle_{ {\rm ss}} \to
T_-$. Thanks to \eqref{eq:p1}  {and \eqref{eq:vanishbd1}} we deduce  $\langle p_1^2 \rangle_{ {\rm ss}} \to
T_-$, which in turn proves \eqref{eq:limitp}. A similar argument proves
(\ref{eq:limitp2}).
Indeed, from \eqref{eq:momentrel3} and Proposition
\ref{prop:current-stationary},  we get
$\langle p_n^2 \rangle_{ {\rm ss}} \to T_+$ and from \eqref{eq:pn-1}  {and \eqref{eq:vanishbd2}},  we deduce
$\langle p_{n-1}^2 \rangle_{ {\rm ss}} \to T_+$.

\subsubsection*{Proofs of (\ref{eq:limit12}) and (\ref{eq:limit11})}
The limit  \eqref{eq:limit12} follows directly from \eqref{022109} and
Proposition \ref{prop:rnpn}. From
\eqref{eq:statelo} we conclude that 
$
\langle r_n \rangle_{ {\rm ss}} \to \bar\tau_+.
$
Using then \eqref{eq:rn}  together with Proposition \ref{prop:rnpn} 
we conclude  \eqref{eq:limit11}.

\subsubsection*{Proofs of (\ref{eq:limit13})  {and (\ref{eq:limit14})}}
Note that
\begin{align}
n^{-2} L(r_1p_1) & = (p_1-p_0)p_1 + (r_2 - r_1)r_1 - \gamma r_1 p_1 \label{eq:r1p1}\\
n^{-2} L(r_n p_{n-1})& = (p_n-p_{n-1})p_{n-1} + (r_n-r_{n-1})r_n -  \gamma r_np_{n-1}. \label{eq:rnpn}
\end{align}
Taking the average with respect to the stationary state, and using
Proposition \ref{prop:rnpn} together with \eqref{eq:limitp} proved
above, we get 
\begin{equation}
\langle r_1^2\rangle_{ {\rm ss}} - \langle r_1 r_2\rangle _s \xrightarrow[n\to\infty]{}\ T_- \label{eq:r1r2}
\end{equation}
and
\begin{equation}
\langle r_n^2\rangle_{ {\rm ss}} - \langle r_{n-1} r_n\rangle _s \xrightarrow[n\to\infty]{}\ T_+. \label{eq:rnrn-1}
\end{equation}
Using the already proved limits (\ref{eq:limit12})  and (\ref{eq:limit11}) we conclude
\eqref{eq:limit13} and \eqref{eq:limit14}. 
\end{proof}

\begin{proposition}[The stationary current: Part II] \label{prop:limit-current}
The following limits hold:
\begin{align}
\label{031403-19}
\phi(1) & \xrightarrow[n\to\infty]{} -\frac{1}{2}(\gamma^{-1}+\gamma)T_-\\
\phi(n) & \xrightarrow[n\to\infty]{}  -\frac{1}{2}(\gamma^{-1}+\gamma)T_+ - \frac{\bar\tau_+^2}{2}. \label{031403-19a}
\end{align}
In consequence, \eqref{eq:limitcurrent} holds and Theorem \ref{theo:current} is proved. 
\end{proposition}

\begin{proof}
Limits in \eqref{031403-19} and \eqref{031403-19a} follow from formula \eqref{eq:phi1},
Proposition  \ref{prop:rnpn} and
the limits computed in Corollary \ref{prop:limit-stationary}. The
limit  \eqref{eq:limitcurrent} is a consequence of \eqref{031403-19}, \eqref{031403-19a} and formula \eqref{eq:njs}.
\end{proof}

\subsection{Energy bounds}\label{sec:bounds-energy}

We now provide bounds on the total energy under the stationary state:

\begin{proposition}[Energy bounds] 
\label{prop:bound-energy} There exists $C=C(\gamma,\tilde\gamma,\bar\tau_+,T_+,T_-)>0$, such that 
\begin{equation}
\label{051403-19}
  \frac1n\sum_{x=1}^n \langle p_x^2 \rangle_{ {\rm ss}} \leq C \qquad
  \text{and}\qquad    \frac1n \sum_{x=1}^n  \langle r_x^2
  \rangle_{ {\rm ss}} \leq C ,\quad n\ge1. \end{equation}
\end{proposition}
\begin{proof}
From the current decomposition given by \eqref{eq:gradientj}, we easily get that 
\[
\phi(x) = (x-1) \bar j_s + \phi(1), \qquad \text{for any } x \in \T_n.
\]
Summing over $x$, this gives
\[
\frac1n\sum_{x=1}^n \phi(x) = \frac1n \sum_{x=2}^n (x-1) \bar j_s + \phi(1)  = \frac{n(n-1)}{2n}\; \bar j_s +  \phi(1).
\]
Therefore, recalling \eqref{031403-19} and \eqref{eq:limitcurrent},
we get
\begin{equation}
\label{eq:limitphi}
\frac1n\sum_{x=1}^n \phi(x) \xrightarrow[n\to\infty]{}  -\frac14\big( \gamma^{-1} + \gamma \big)(T_++T_-) - \frac{\bar\tau_+^2}{4\gamma}.
\end{equation}
From \eqref{eq:phi1}, we have 
\begin{equation}\label{eq:aaa}
\frac1n\sum_{x=1}^n \phi(x)  = - \frac{1}{2\gamma n} \sum_{x=1}^n \langle r_x^2 \rangle_{ {\rm ss}} - \frac{1}{2\gamma n} \sum_{x=1}^n \langle p_x p_{x-1} \rangle_{ {\rm ss}} - \frac{\gamma}{2n} \sum_{x=1}^n \langle p_x^2 \rangle_{ {\rm ss}} + \frac{\gamma}{4n} \big( \langle p_n^2 \rangle_{ {\rm ss}}- \langle p_0^2 \rangle_{ {\rm ss}}\big). \end{equation}
To compute the limit of the second sum in the right hand side of
\eqref{eq:aaa}, we first write: 
\begin{align}
\label{023108-19}
n^{-2} L (p_{x-1}p_x) & = (r_{x+1}-r_x) p_{x-1} + (r_x-r_{x-1})p_x - 2\gamma p_xp_{x-1}, \qquad x \in\{2,\dots,n-1\}.
\end{align}
Thus,  taking the average with respect to the stationary state and
subsequently using \eqref{eq:equ}, we obtain 
\begin{align}
2\gamma \langle p_xp_{x-1}\rangle_{ {\rm ss}} & = \langle p_x r_x\rangle_{ {\rm ss}}
+\langle p_{x-1}r_{x+1} \rangle_{ {\rm ss}} -  \langle p_{x}r_{x-1}\rangle_{ {\rm ss}}  -
                                      \langle
                                      p_{x-1}r_{x}\rangle_{ {\rm ss}} \notag\label{eq:gradp1}
  \\
&=
\langle p_{x-1}r_{x+1} \rangle_{ {\rm ss}} -  \langle p_{x}r_{x-1}\rangle_{ {\rm ss}} . 
\end{align}
On the other hand
\begin{align*}
n^{-2} L (r_xr_{x+1}) & = (p_x-p_{x-1})r_{x+1} + (p_{x+1}-p_x)r_x , \qquad x \in \{1,\dots,n-1\} .
\end{align*}
Hence, taking the average and using again \eqref{eq:equ}, we get
\begin{align*}
0& = \langle p_{x+1}r_{x} \rangle_{ {\rm ss}}+\langle p_{x}r_{x+1} \rangle_{ {\rm ss}} -  \langle p_{x}r_{x}\rangle_{ {\rm ss}}  -  \langle p_{x-1}r_{x+1}\rangle_{ {\rm ss}} \notag \\
& =\langle p_{x+1}r_{x} \rangle_{ {\rm ss}}+\langle p_{x+1}r_{x+1} \rangle_{ {\rm ss}} -  \langle p_{x}r_{x}\rangle_{ {\rm ss}}  -  \langle p_{x-1}r_{x+1}\rangle_{ {\rm ss}},  
\end{align*}
which yields
$$
\langle p_{x+1}r_{x+1} \rangle_{ {\rm ss}} -  \langle p_{x}r_{x}\rangle_{ {\rm ss}}  =  \langle p_{x-1}r_{x+1}\rangle_{ {\rm ss}}-\langle p_{x+1}r_{x} \rangle_{ {\rm ss}}
$$
for any $x \in \{2,\dots,n-1\}$. Combining with \eqref{eq:gradp1} we
get
\begin{align}
\label{eq:gradpa}
2\gamma \langle p_xp_{x-1}\rangle_{ {\rm ss}} & = 
\langle p_{x+1}r_{x+1} \rangle_{ {\rm ss}} -  \langle p_{x}r_{x}\rangle_{ {\rm ss}}+\langle p_{x+1}r_{x}\rangle_{ {\rm ss}}-  \langle p_{x}r_{x-1}\rangle_{ {\rm ss}} ,                         \end{align}
for any $ x \in \{2,\dots,n-1\}$. 
 Summing over $x$, one gets: 
\begin{equation} \sum_{x=2}^{n-1} \langle p_{x}p_{x-1}\rangle_{ {\rm ss}} = \frac{1}{2\gamma} \big(  \langle p_n r_n \rangle_{ {\rm ss}} - \langle p_2r_2 \rangle_{ {\rm ss}} + \langle p_n r_{n-1} \rangle_{ {\rm ss}} - \langle p_2 r_1 \rangle_{ {\rm ss}}\big). \label{eq:sump} \end{equation}
To compute the limit as $n\to\infty$, we need to estimate the
covariances appearing in the right hand side.
The covariance $\langle
p_n r_n \rangle_{ {\rm ss}}  $ can be estimated thanks to Proposition
\ref{prop:rnpn}.
 We still need the bounds on the covariances $\langle p_2r_2 \rangle_{ {\rm ss}}$,  $ \langle p_n r_{n-1}
\rangle_{ {\rm ss}} $ and $ \langle p_2 r_1 \rangle_{ {\rm ss}}$. 
To deal with it write
\begin{align}
n^{-2}L(p_0p_1)&  = (r_2-r_1)p_0 + r_1 p_1 - \frac12(3\gamma + \tilde\gamma) p_0p_1 \label{eq:b1}\\
n^{-2}L(r_1r_2)&  = (p_1-p_0)r_2 + (p_2 -p_1) r_1 \label{eq:b2} \\
n^{-2}L(p_{n-1}p_n) & = (\bar\tau_+-r_n)p_{n-1} + (r_n-r_{n-1})p_n - \frac12(3\gamma + \tilde\gamma)  p_{n-1}p_n. \label{eq:b3}
\end{align}
Taking averages  {with respect to the stationary state} and summing  \eqref{eq:b1} and \eqref{eq:b2}
sideways gives (using $\langle p_2 r_2 \rangle_{ {\rm ss}} = \langle p_1 r_2 \rangle_{ {\rm ss}}$ from \eqref{eq:equ})  
\begin{equation}\label{eq:p2r2} \langle p_2 r_2 \rangle_{ {\rm ss}} + \langle p_2 r_1 \rangle_{ {\rm ss}} = \langle p_0 r_1 \rangle_{ {\rm ss}} + \frac12(3\gamma + \tilde \gamma) \langle p_0 p_1 \rangle_{ {\rm ss}} \xrightarrow[n\to\infty]{}0, \end{equation}
 from Proposition \ref{prop:rnpn}. Moreover, \eqref{eq:b3} gives (using $\langle p_n r_n \rangle_{ {\rm ss}} = \langle p_{n-1} r_n \rangle_{ {\rm ss}}$)  
 \begin{equation} \langle r_{n-1}p_n \rangle_{ {\rm ss}} = \bar\tau_+ \langle p_{n-1}\rangle_{ {\rm ss}} - \frac12(3\gamma+\tilde\gamma) \langle p_{n-1}p_n\rangle_{ {\rm ss}} \xrightarrow[n\to\infty]{}0, \label{eq:pnrn-1}\end{equation}
 from \eqref{eq:statvel} and Proposition \ref{prop:rnpn}. Therefore,
 we have proved that \eqref{eq:sump} vanishes as $n\to\infty$. In
 fact, due to the estimates obtained in Proposition \ref{prop:rnpn} we have even proved that there exists a constant $C=C(\gamma,\tilde\gamma,\bar\tau_+,T_+,T_-)>0$, such that
 \begin{equation}
\label{041403-19}
\Big|\sum_{x=1}^n \langle p_x p_{x-1}\rangle_{ {\rm ss}}\Big| \leq
\frac{C}{\sqrt n} ,\quad n\ge1.
\end{equation}
From \eqref{eq:aaa} it follows that 
\begin{align*}
\frac{1}{2\gamma n} \sum_{x=1}^n \langle r_x^2 \rangle_{ {\rm ss}} + \frac{\gamma}{2n} \sum_{x=1}^n \langle p_x^2 \rangle_{ {\rm ss}} & = - \frac{1}{2 \gamma n }\sum_{x=1}^n \langle p_x p_{x-1} \rangle_{ {\rm ss}} - \frac{1}{ n} \sum_{x=1}^n \phi(x) + \frac{\gamma}{4 n} \big(\langle p_n^2\rangle_{ {\rm ss}} - \langle p_0^2 \rangle_{ {\rm ss}} \big) \\
& \xrightarrow[n\to\infty]{} \frac14\big( \gamma^{-1} + \gamma \big)(T_++T_-)+ \frac{\bar\tau_+^2}{4\gamma},
\end{align*}
due to \eqref{eq:limitphi} and \eqref{041403-19}.
This in particular implies estimate \eqref{051403-19}.
\end{proof}

Thanks to the energy bounds, we  are finally able to prove one further
convergence, which will be essential in establishing  Proposition \ref{prop:limitsE}. 
\begin{proposition} \label{prop:correlation} 
For any continuous test function $G:\T\to\R$ we have
\begin{equation}
\frac1n \sum_{x=1}^{n-1}G\Big(\frac x n\Big)\langle p_x p_{x+1}\rangle_{ {\rm ss}}   \xrightarrow[n\to\infty]{} 0. \label{eq:corr2}
\end{equation}
\end{proposition}

\begin{proof}
Assume first that $G\in C^1(\T)$. For the brevity sake 
we denote $G_x:=G( x/ n)$ for any $x \in \T_n$ and  $\psi(x) = \langle p_{x+1} r_{x+1} \rangle_{ {\rm ss}} + \langle p_{x+1} r_{x}\rangle_{ {\rm ss}}.$ Then \eqref{eq:gradpa}  says that 
\[ \langle p_xp_{x+1} \rangle_{ {\rm ss}} = \frac{1}{2\gamma}\big(\psi(x+1)-\psi(x)\big), \qquad \text{for any } x \in \{1,\dots,n-2\}.\]
Therefore, by an application of summation by parts formula, we get
\begin{align}
\frac1n \sum_{x=1}^{n-1} G_x \langle p_x p_{x+1}\rangle_{ {\rm ss}} & = \frac{1}{2\gamma n^2} \sum_{x=2}^{n-2} n (G_{x-1}-G_x) \psi(x) \label{eq:sum}\\
& \quad + \frac{1}{n}\langle p_np_{n-1} \rangle_{ {\rm ss}} G_{n-1} + \frac{1}{2\gamma n} \big( \langle p_n r_n \rangle_{ {\rm ss}} + \langle p_n r_{n-1} \rangle_{ {\rm ss}} \big) \label{eq:bndterm1} \\ & \quad - \frac{1}{2\gamma n} \big(\langle p_2 r_2 \rangle_{ {\rm ss}} + \langle p_2 r_1\rangle_{ {\rm ss}} \big) G_1.  \label{eq:bndterm2}
\end{align}
The boundary terms \eqref{eq:bndterm1} and \eqref{eq:bndterm2} vanish,
as $n \to \infty$, thanks to \eqref{eq:vanishbd2}, \eqref{eq:p2r2} and
\eqref{eq:pnrn-1}. To deal with the sum in the right hand side of \eqref{eq:sum} 
note that, since $G\in C^1(\T)$, we have
\begin{equation}
\label{gp}
\sup_{x\in\{2,\ldots,n-2\}}n |G_{x-1}-G_x|\le \|G'\|_\infty.
\end{equation}
Since 
$$
|\psi(x)|\le 2\left(2\langle p^2_{x+1} \rangle_{ {\rm ss}}+\langle r_x^2
  \rangle_{ {\rm ss}}+\langle r^2_{x+1} \rangle_{ {\rm ss}}\right), \qquad  x\in\{1,\ldots,n-1\}
$$
we conclude that 
\begin{equation}
\label{gp1}
\frac{1}{2\gamma n^2} \left|\sum_{x=2}^{n-2} n (G_{x-1}-G_x)
  \psi(x)\right|\le \frac{C}{n}\left\{ \frac1n\sum_{x=1}^n \big( \langle p_x^2 \rangle_{ {\rm ss}}+ \langle r_x^2 \rangle_{ {\rm ss}}\big)\right\},
\end{equation}
which vanishes, as $n\to+\infty$, thanks to the energy bound \eqref{051403-19}. This proves
\eqref{eq:corr2} for any test function $G\in C^1(\T)$. The result can
be extended to all continuous functions by the standard density argument and
the energy bound \eqref{051403-19}.
\end{proof}

\subsection{Proof of Proposition \ref{prop:limitsE}} \label{sec:proofprop}
We now have at our disposal all components needed to prove Proposition \ref{prop:limitsE} and thus conclude the proof of Theorem \ref{theo:energy}.  There are three convergences to prove: 

\subsubsection*{Proof of \eqref{eq:limitEphi}}
From Proposition \ref{prop:current} (in particular
\eqref{eq:phiharmonic}),  \mm{the function $\phi$ is linear and is completely determined from the values at the endpoints. More precisely, 
\[ \phi(x)= \frac{\phi(n)-\phi(1)}{n-1} x + \frac{n \phi(1)-\phi(n)}{n-1}, \qquad \text{for any } x \in \T_n.\]
Since, from Proposition \ref{prop:limit-current}, the values $\phi(1)$ and $\phi(n)$ are of order 1 as $n\to\infty$, we see that
\[ 
\phi(x) \simeq \big(\phi(n)-\phi(1)\big) \frac{x}{n} + \phi(1), \qquad \text{as } n\to\infty,
\]} and therefore we easily obtain
\begin{equation*} \frac{1}{n}  \sum_{x \in \T_n} G\Big(\frac{x}{n}\Big) \phi(x) \xrightarrow[n\to\infty]{} \int_\T   G(u) \bigg\{ -\frac{\bar\tau_+^2}{2\gamma} \; u - \frac12(\gamma^{-1}+\gamma)\Big[(T_+-T_-)u+ T_-\Big] \bigg\} \; \dd u.
 \end{equation*}
which proves directly \eqref{eq:limitEphi}.

\subsubsection*{Proof of \eqref{eq:limitEnabla}}
Concerning $\mc H_n^\nabla(G)$ we use  a summation by parts formula
(with the notation $G_x=G( x /n)$), which leads to: 
\begin{equation*}
\mc H_n^\nabla(G) = \frac{\gamma^2}{2(1+\gamma^2)} \bigg( \frac{G_n}{n}\langle p_n^2 \rangle_{ {\rm ss}} - \frac{G_1}{n}\langle p_0^2 \rangle_{ {\rm ss}} + \frac{1}{n^2} \sum_{x=1}^{n-1} n(G_x-G_{x+1}) \langle p_x^2 \rangle_{ {\rm ss}} \bigg).
\end{equation*} 
The boundary terms in the right hand side vanish, as $n\to+\infty$,
since $\langle p_n^2\rangle_{ {\rm ss}}$ and $\langle p_0^2 \rangle_{ {\rm ss}}$ are
bounded, due to \eqref{eq:10}. To deal with the limit of the last sum in the right
hand side, we can repeat the argument made in \eqref{gp}-\eqref{gp1},
which shows that the expression vanishes. Thus 
\eqref{eq:limitEnabla} holds. 

\subsubsection*{Proof of \eqref{eq:limitEcorr}} 
This is a consequence of  \eqref{eq:corr2}.


\appendix

\section{Non-Stationary behaviour}

\label{sec:form-deriv-equat}

In this section we explain how to derive \eqref{eq:linear}
  and \eqref{eq:4}: while the derivation of \eqref{eq:linear}
is rigorous, in order to obtain \eqref{eq:4} we need to assume a form of local equilibrium that
allows for the local equipartition of  kinetic and potential energy,
see \eqref{conj1} below. In the stationary setting this term corresponds to $\mc
H_n^{\text{m}}(G)$ in \eqref{energy:decomp} and, similarly, does not appear in the
case $\ga=1$. 
Unfortunately, quite analogously with the stationary situation, the
relative entropy method does not allow us to treat the case $\gamma \not= 1$.
Throughout the present section we allow $\bar
\tau_+(t)$ to be a $C^1$ function.

\subsection{Preliminaries}

 In the present section we establish non-stationary asymptotics
corresponding to Corollary \ref{prop:limit-stationary}. They will be
useful in proving the hydrodynamic limit in Section \ref{secA.2}.
Since  $\nu_{T_+}$ is not  stationary, except for the corresponding equilibrium boundary conditions,
the relative entropy $H_n(t)$ defined as \[ 
H_n(t):= \int f_n(t) \log f_n(t) \dd \nu_{T_+}
\] is not strictly decreasing in time, where hereafter $f_n(t)$ is
the density of the $\Om_n$--valued  random variable
$\left(\mathbf{r}_n(t),\mathbf{p}_n(t)\right)$ (recall \eqref{eq:defrp}), with respect to $\nu_{T_+}$. However, the effect of the boundary condition can be controlled and one can
obtain a linear in $n$ bound at any  time $t$, i.e.~(see
the proof of
Proposition 4.1 in \cite{letizia} for the details of the argument)
\begin{equation}
  \label{eq:5}
  H_n(t) \le  C(t) n,\quad n\ge 1,\,t\ge0.
\end{equation}
Both here and throughout the remainder of the paper  $C(t)$ shall denote a
  generic constant,   always independent of
  $n$ and locally bounded in $t$.

Furthermore, one obtains the bounds on the Dirichlet form controlling the entropy production,
similar to \eqref{eq:88},
\begin{multline}
 \int_0^t \dd s \left[ \gamma \sum_{x=0}^{n-1} {\mc D}_x (f_n(s)) + 
 \tilde\gamma T_- \; \Big\llangle   \frac{(\partial_{p_0} (f_n(s)/h))^2}{(f_n(s)/h)} \Big\rrangle_{T_-}+ 
\tilde\gamma T_+ \; \Big\llangle  \frac{(\partial_{p_n} f_n(s))^2}{f_n(s)}
\Big\rrangle_{T_+} \right] \\ \leq \frac{C(t)}{n}, \label{eq:88t}
\end{multline}
where $h = g_{T_-}/g_{T_+}$  and $n\ge1$, $t\ge0$,  see
  Proposition 4.1 and Appendix D of \cite{letizia} for the
  proof.

Below we list some consequences of the
  above bounds on the entropy and Dirichlet forms.
\begin{lemma}\label{pbond}
 The following equalities hold:
  \begin{equation}
    \label{eq:ex-1}
\lim_{n\to\infty}    \mathbb{E}\bigg[   \int_0^t  \left(p_{n}^2(s) - T_+\right)\dd s\bigg]  = 0,\qquad
   \lim_{n\to\infty}  \mathbb{E}\bigg[   \int_0^t  \left(p_{0}^2(s) - T_-\right)\dd s\bigg]  = 0
 \end{equation}
and
 \begin{equation}
   \label{eq:ex-2}
   \lim_{n\to\infty}    \mathbb{E}\bigg[   \int_0^t  j_{n-1,n}(s) \dd s\bigg]  = 0,\qquad
   \lim_{n\to\infty}    \mathbb{E}\bigg[   \int_0^t  j_{0,1}(s) \dd s\bigg]  = 0.
 \end{equation}

\end{lemma}

\begin{proof}
  Note that, 
  \begin{align*}
      \mathbb{E}\bigg[   \int_0^t  \left(p_{n}^2(s) - T_+\right)\dd s\bigg] &  =
      \int_0^t  \dd s \int \left(p_{n}^2- T_+\right) f_n(s) \dd\nu_{T_+} \\ & =
      T_+  \int_0^t  \dd s \int p_{n} \partial_{p_n} f_n(s) \dd\nu_{T_+} .
  \end{align*}
Thus, by the Cauchy-Schwarz inequality
\begin{multline}
  \left|\mathbb{E}\bigg[   \int_0^t  \left(p_{n}^2(s) - T_+\right)\dd s\bigg]\right|
 \\  \le T_+ \left( \int_0^t  \dd s \int p_{n}^2 f_n(s) \dd\nu_{T_+} \right)^{1/2}
      \left( \int_0^t  \dd s \int \frac{\left( \partial_{p_n} f_n(s) \right)^2}{f_n(s)} \dd\nu_{T_+} \right)^{1/2}.
    \label{eq:7}
\end{multline}
  By the entropy inequality, see e.g.~
  \cite[p. 338]{KL}, we can write (recall \eqref{Hn})
\begin{multline*}
  \int  \mathcal{H}_n (\mathbf r, \mathbf p)  f_n(s, \mathbf r, \mathbf p) \nu_{T_+} (\dd\mathbf r, \dd\mathbf p) \\ \le \frac{1}{\al}\left\{\log\left[\int
     \exp\left\{\al\mathcal{H}_n (\mathbf r, \mathbf p) \right\}\nu_{T_+} (\dd\mathbf r, \dd\mathbf p)\right]+H_n(s)\right\}
  \end{multline*}
for any $\al>0$.
 By \eqref{eq:5}, for any $t>0$ and a sufficiently small $\al>0$, there exists $C>0$
 such that
  \begin{equation}
\label{p-n}
  \sup_{s\in[0,t]}  \int  \mathcal{H}_n (\mathbf r, \mathbf p)  f_n(s, \mathbf r, \mathbf p) \nu_{T_+} (\dd\mathbf r, \dd\mathbf p) \le  Cn,\quad n\ge 1.
  \end{equation}
  Consequently, by \eqref{eq:88t} and \eqref{eq:7}, there exists $C>0$
  such that 
  $$
\sup_{s\in[0,t]} \left|\mathbb{E}\left[  \int_0^t  \left(p_{n}^2(s) - T_+\right)\dd
   s\right]\right| \le C,\quad n\ge1.
$$
 Hence, in particular we obtain
  \begin{equation}\label{eq:p2b}
  \int_0^t \dd s\ \int p_{n}^2 f_n(s) \dd\nu_{T_+} \le C.
  \end{equation}
  Using this estimate in  \eqref{eq:7} together with  \eqref{eq:88t}
  we
  conclude that for any $t\ge0$ there exists $C>0$ for which
  \begin{equation}
  \left|\mathbb{E}\bigg[   \int_0^t  \left(p_{n}^2(s) - T_+\right)\dd s\bigg]\right|
  \le \frac{C}{\sqrt n},\quad n\ge1.
    \label{eq:7tk}
\end{equation}
 Hence the first equality of \eqref{eq:ex-1} follows.
 The proof of the second equality of  \eqref{eq:ex-1} is similar.

\medskip

Concerning \eqref{eq:ex-2}: from the energy
 conservation it follows that
 \begin{equation*}
   n^{-2} L \mathcal E_n = j_{n-1,n}+\frac{\tilde\gamma}{2}  (T_+ - p_n^2) + \bar\tau_+(t) p_n.
 \end{equation*}
To deal with the  term $\bar\tau_+(t) p_n$, note that
\begin{align}
    \left|\int_0^t \bar\tau_+ (s)  \mathbb{E} \big[p_n(s)\big] \dd
   s\right| & \le\|\bar\tau_+\|_\infty \int_0^t \left|\int p_{n} f_n(s)
   \dd\nu_{T_+} \right|\dd s \notag\\
& =\frac{\|\bar\tau_+\|_\infty}{T_+} \int_0^t\left|\int \partial_{p_{n}} f_n(s) \dd\nu_{T_+} \right|\dd s
 \notag\\
& \le \frac{\|\bar\tau_+\|_\infty}{T_+} \int_0^t \left(\int  f_n(s)
   d\nu_{T_+} \right)^{1/2}  \left( \int \frac{\left( \partial_{p_n}
       f_n(s) \right)^2}{f_n(s)} \dd\nu_{T_+} \right)^{1/2}\dd s \notag\\
      & \le  \frac{\|\bar\tau_+\|_\infty}{T_+}
    \frac{C}{  \sqrt{n}} \mathop{\longrightarrow}_{n\to\infty} 0,  \label{eq:1a}
    \end{align}
by virtue of \eqref{eq:88t}.

 The first equality of \eqref{eq:ex-2} is then a direct consequence of
 \eqref{eq:ex-1}. The same argument works for $j_{0,1}$. 
\end{proof}

Using the energy conservation it follows immediately:
\begin{corollary}
{The currents, defined in \eqref{eq:current} and
  \eqref{eq:current-bound},  satisfy}
  \begin{equation}
    \label{eq:11}
      \lim_{n\to\infty}    \mathbb{E}\bigg[   \int_0^t  j_{x,x+1}(s)
      \dd s\bigg]  = 0, \quad x=-1,\dots,n,\quad t\ge0.
  \end{equation}
\end{corollary}

Concerning the potential energy at the boundary
  points we have the following bound.
\begin{lemma}
  \label{lem-rb}
  There exists a  constant $C<\infty$ such that
  \begin{equation}
    \label{eq:13}
    \mathbb{E}\left[ \int_0^t \left(r_1^2(s) + r_n^2(s)\right) \;
      \dd s\right] \le C,\quad{\color{blue} n\ge1.}
  \end{equation}
\end{lemma}

\begin{proof} Using \eqref{eq:b} we get
\begin{align}
&     n^{-2} \bbE\big[p_n(t) r_n(t)-p_n(0) r_n(0)\big]\label{eq:b111} \\
&= \int_0^t \bbE\left[p_n(s)(p_n(s) - p_{n-1}(s)) + (\bar\tau_+(s)-r_n(s)) r_n(s) - \frac12(\tilde\gamma +\gamma)
    p_n(s) r_n(s) \right]\dd s. \notag
  \end{align}
The term in the left hand side vanish, as $n\to+\infty$, due to
estimate \eqref{p-n}.
 We can repeat then
  the same arguments as we have used to obtain 
  \eqref{eq:rn-2} and conclude that there exists $C>0$ such that
  \begin{equation}\label{eq:rn-21}
  \mathbb{E}\left[ \int_0^t r_n^2(s)\dd s\right]  \le   C
  \left\{\mathbb{E}\left[ \int_0^t p_n^2(s)\dd s\right] +1\right\},\quad
  n\ge 1.
\end{equation} 
Estimate \eqref{eq:p2b} can be used to obtain the desired bound for $\mathbb{E}\left[ \int_0^t r_n^2(s) \;
      \dd s\right] $. An analogous  estimate on $\mathbb{E}\left[ \int_0^t r_1^2(s) \;
      \dd s\right] $ follows from the same argument, using \eqref{eq:a}
    and the second equality in \eqref{eq:ex-1} instead.
\end{proof}

\begin{lemma}\label{lem:bound2}
  The following convergences hold: at the left boundary point:
  \begin{align}
     \mathbb{E}\left[ \int_0^t p_{0}(s) r_1(s) \dd s\right] & \xrightarrow[n\to\infty]{}
                                                                                  0 \label{ex-2-l}\\
    \mathbb{E}\left[ \int_0^t  \left(p_{1}^2(s) - p_{0}^2(s)\right)\dd s\right] & \xrightarrow[n\to\infty]{}
                                                                                  0 \label{ex-2}\\
 \mathbb{E}\left[  \int_0^t r_{1}(s) r_{2}(s)\; \dd s\right]  &\xrightarrow[n\to\infty]{}  0 \label{ex-2ll}
\end{align}
and at the right boundary point:
\begin{align}
   \mathbb{E}\left[ \int_0^t p_{n-1}(s) r_n(s) \dd s\right] & \xrightarrow[n\to\infty]{}
                                                                                  0 \label{ex-2-r}\\
 \mathbb{E}\left[ \int_0^t  \left(p_{n-1}^2(s) - p_{n}^2(s)\right)\dd s\right] &\xrightarrow[n\to\infty]{}  0 ,\label{ex-2n}\\
  \mathbb{E}\left[  \int_0^t  \left( r_{n-1}(s) r_{n}(s) - \bar\tau_+^2(s)\right)\dd s\right] &\xrightarrow[n\to\infty]{} 0 .
  \label{ex-2r}
\end{align}
\end{lemma}
\begin{proof}
Since
  $ n^{-2}Lr_n^2 = 2 r_n p_n - 2r_n p_{n-1}$,
 using
  \eqref{p-n} we conclude that \eqref{ex-2-r} holds, provided
  that we can prove
\begin{align}
   \mathbb{E}\left[ \int_0^t r_n(s) p_{n}(s) \dd s\right] & \xrightarrow[n\to\infty]{}
                                                                                  0. \label{ex-2-r1}
\end{align}
 The latter is a consequence of the following estimate, cf. \eqref{eq:88t}, 
  \begin{align*}
    \left|\mathbb{E}\left[ \int_0^t r_n(s) p_n(s) \dd s \right]\right|
      & = \left|\int_0^t \dd s \int r_n \partial_{p_n} f_n(s)
        \dd\nu_{T_+}\right| \\
 & \le  \left(\int_0^t \dd s \int r_n^2 f_n(s) \dd\nu_{T_+} \right)^{1/2}  \left( \int_0^t  \dd s \int \frac{\left( \partial_{p_n} f_n(s) \right)^2}{f_n(s)} \dd\nu_{T_+} \right)^{1/2}\\
     & \le  \left(\int_0^t \dd s \int r_n^2 f_n(s) \dd\nu_{T_+} \right)^{1/2} 
     \frac{C}{ \sqrt{n}} \xrightarrow[n\to\infty]{} 0.
  \end{align*}
  The proof of \eqref{ex-2-l} is similar.
\medskip

  To show \eqref{ex-2n}
note that (see \eqref{eq:p2-p2}) 
\begin{align*}
\gamma \int_0^t  \bbE\left[p_{n-1}^2 (s)- p_n^2(s)\right] \dd s =  &  \; 2
  \int_0^t \bar\tau_+ (s) \bbE \big[ p_n (s)\big]  \dd s - 2
  \int_0^t  \bbE\big[ r_n(s) p_n(s) \big] \dd s\\
&
+  \tilde\gamma \int_0^t  \bbE\left[T_+-p_n^2(s) \right] \dd s -\frac{1}{n^{2}}\bbE\left[p_n^2(t)-p_n^2(0)\right] . 
\end{align*}
 The second, third and fourth terms in the right hand side vanish due
 to 
  \eqref{ex-2-r1}, Lemma \ref{pbond} and \eqref{p-n}, respectively.
The first term has been already treated in \eqref{eq:1a}. An analogous argument, starting from
\eqref{eq:p1} allows us to prove \eqref{ex-2}.

\medskip

  Besides, we have
  \begin{multline}
    \label{eq:9}
     \mathbb{E}\left[ \int_0^t p_{n-1}(s) p_n(s)
       \dd s\right] \\ =\mathbb{E}\left[ \int_0^t p_{n-1}(s) (p_n(s)-T_+)
       \dd s\right]+T_+\mathbb{E}\left[ \int_0^t p_{n-1}(s) \dd s\right]
      \xrightarrow[n\to\infty]{} 0.
  \end{multline}
The above convergence is proved as follows: the first term in the right hand side can be estimated by the
Cauchy-Schwarz inequality. Then we can use the  bound 
$$ 
\mathbb{E}\left[ \int_0^t p_{n-1}^2(s) \dd s\right] \le C, \quad n\ge1
$$ (it follows from the already proved \eqref{ex-2n}) and Lemma
\ref{pbond} to prove that it vanishes, as $n\to+\infty$.
To show that the second term vanishes we can use estimates analogous
to \eqref{eq:1a}. The argument for \eqref{ex-2-l} follows essentially
the same lines.

  By \eqref{eq:b} we can now write
  \begin{align}\label{eq:bt}
   \mathbb{E}\left[  \int_0^t  \left(r_{n}^2(s) - T_+ -
     \bar\tau_+^2(s)\right)\dd s\right] = & \; \int_0^t\bar\tau_+ (s)
     \mathbb{E} \big[r_n(s) -\bar\tau_+ (s)\big]\dd s 
     \notag \\ &  +   \int_0^t\mathbb{E}\big[  p_n^2 (s)-
     T_+  \big] \dd s 
- \int_0^t\mathbb{E}\big[  p_n (s)p_{n-1}(s)
     \big] \dd s \notag \\ & - \frac12 (\tilde\gamma+ \gamma) \int_0^t\mathbb{E}\big[  p_n (s)r_n(s) \big] \dd s \notag \\ & - \frac{1}{n^{2}} \mathbb{E}\big[   p_n(t) r_n(t)-p_n(0) r_n(0)\big].
  \end{align}
 By the previous results, the last four terms in the right hand side
 vanish, as $n\to+\infty$. Concerning the first term we use
\begin{equation}
\label{012308-19}
n^{-2}Lp_n(s)=\bar\tau_+ (s)-r_n(s) -\frac12 (\tilde\gamma+ \gamma)p_n(s)
\end{equation}
to conclude that
\begin{align}\label{eq:bt1}
    \int_0^t\bar\tau_+ (s)
     \mathbb{E} \big[r_n(s) -\bar\tau_+ (s)\big]\dd s=   & \; \frac12
     (\tilde\gamma+ \gamma)\int_0^t\bar\tau_+ (s)\mathbb{E}\big[p_n (s)\big]\dd
     s  \notag \\ & - \frac{1}{n^{2}}\int_0^t\bar\tau_+ (s) \frac{\dd}{\dd s}\mathbb{E}
    \big[ p_n(s)\big]\dd s.
  \end{align}
The first term vanishes, as $n\to+\infty$, by  \eqref{eq:1a}. By
integration by parts  the second term equals
$$
\frac{1}{n^{2}}\int_0^t\bar\tau_+' (s) \mathbb{E}
    \big[ p_n(s)\big]\dd s-\frac{1}{n^{2}}\Big(\bar\tau_+ (t) \mathbb{E}
    \big[ p_n(t)\big]-\bar\tau_+ (0) \mathbb{E}\big[
    p_n(0)\big]\Big)
$$
which also vanishes, thanks to \eqref{eq:1a} and \eqref{p-n}. Therefore,
  \begin{equation*}
     \mathbb{E}\left[  \int_0^t  \left(r_{n}^2(s) - T_+ - \bar\tau_+^2(s)\right)\dd s\right] \longrightarrow 0.
   \end{equation*}
To see  \eqref{ex-2r} 
it suffices to show that
\begin{equation*}
     \mathbb{E}\left[  \int_0^t  \left(r_{n}^2(s) - r_{n}(s) r_{n-1}(s)-T_+\right)\dd s\right] \longrightarrow 0.
   \end{equation*}
For that purpose we invoke \eqref{eq:rnpn}, which permits to write
\begin{align*}
&    \mathbb{E}\left[ \int_0^t   \Big((r_{n-1}(s)-r_{n}(s))r_n(s)+T_+\Big) \dd s\right]
=\mathbb{E}\left[ \int_0^t   (p_n^2(s)-p_{n-1}^2(s))\dd
  s\right] \\
& \qquad
+\mathbb{E}\left[ \int_0^t   p_n(s)p_{n-1}(s) \dd
  s\right]  +\mathbb{E}\left[ \int_0^t   (T_+-p_{n}^2(s))\dd
  s\right]\\
& \qquad
-
\gamma \mathbb{E}\left[ \int_0^t r_n(s)p_{n-1}(s)\dd s\right] -\frac{1}{n^{2}}\bbE \big[r_n (t)p_{n-1}(t)-r_n (0)p_{n-1}(0)\big].
 \end{align*}
Each term in the right hand side of the above equality  vanishes, as
$n\to+\infty$, by virtue of the already proved estimates. 
  With a similar procedure we obtain \eqref{ex-2ll}. 
\end{proof}

\subsection{ Hydrodynamic limit}

\label{secA.2}

Let us now turn to equation \eqref{eq:linear}, which can be formulated in a weak form as:
\begin{multline}
  \label{eq:wlinear}
  \int_0^1 \dd u\; G(u) \big(r(t,u) - r(0,u)\big)  \\ = \frac {1}{{\gamma}} \int_0^t \dd s
  \int_0^1    \dd u \; G''(u) r(s,u)  -  \frac {1}{{\gamma}}  G'(1) \int_0^t \bar\tau_+(s) \dd s
  ,
\end{multline}
for any  test function $G\in C^2_{0,1}(\T)$ -- the class of
$C^2$ functions  on $[0,1]$ such that $G(0) = G(1) =
0$. Existence and uniqueness of such weak solutions in an appropriate
space of integrable functions are standard. 
By the microscopic evolution equations
\eqref{eq:qdynamics} we have  (recall that $G_x=G(x/n)$) 
\begin{align}
  \bbE\left[  \frac 1n \sum_{x=1}^n 
    G_x\left(r_x(t) -  r_x(0)\right) \right] &
    = \bbE\left[ \int_0^t \dd s \ n \sum_{x=1}^n 
   G_x \left(p_x(s) -  p_{x-1}(s)\right) \right]\notag \\
& = \bbE\left[  \int_0^t \dd s \left\{ -\sum_{x=1}^{n-1} 
    (\nabla_n G)_x\;p_x(s) - n
    G_1 p_0(s)\right\}\right],
  \label{eq:timevol1} \end{align}
where $ (\nabla_n G)_x:= n(G_{x+1} - G_x)$. 
Using 
\eqref{eq:pdyn} we can write \eqref{eq:timevol1} as
\begin{align}
  &
 \bbE\left[ -\int_0^t  \dd s  \left\{ \sum_{x=1}^{n-1} \frac{1}{\gamma} (\nabla_n
      G)_x\left(r_{x+1}(s) - r_x(s)\right) +\frac {1}{2(\gamma+ \tilde\gamma)} n
      G_1 r_1(s) \right\}\right]  \label{eq:timevol++}\\
&+\bbE\left[ \frac{1}{\gamma n^2} \sum_{x=1}^{n-1} (\nabla_n
      G)_x\left(p_x(t) - p_x(0)\right) + \frac
      {1}{2(\gamma +\tilde \gamma)n^2}   
    n G_1\left(p_0(t) - p_0(0)\right) \right].
\label{eq:timevol+b} \end{align}
Since $G$ is smooth, $nG_1 \to G'(0)$ and $(\nabla_nG)_x\to G'(x)$, as $n\to +\infty$.
Using this and  \eqref{p-n}, one shows that the expression
\eqref{eq:timevol+b} converges to $0$.
The only significant term  is therefore the first one \eqref{eq:timevol++}. 
Summing by parts, using the notation 
$$
(\Delta_n G)_x:=
n^2(G_{x+1}+G_{x-1}-2G_x)
$$ and recalling  that $G(0)=0$, it can be rewritten as
\begin{multline}\label{eq:timevol3}
\bbE\left[ \int_0^t  \frac{1}{\gamma}\left\{\frac 1n \sum_{x=2}^{n-1}
      (\Delta_n G)_x\; r_x(s) - (\nabla_n G)_{n-1}\; r_n(s) \right\} \dd s\right]\\
      - \left(\frac {1}{2(\gamma+ \tilde\gamma)}  
      (\nabla_n G)_0 -\frac{1}{\gamma}  (\nabla_n G)_1 \right)
   \bbE\left[\int_0^t r_1(s) \dd s\right].    
\end{multline}
It is easy to see,  using \eqref{eq:rbd}, that 
\begin{align*}
\lim_{n\to+\infty}\bbE\left[\int_0^t r_1(s) \dd s\right] & =
 \frac{\ga+\tilde\ga}{2}\lim_{n\to+\infty} \int_0^t  \dd s\int
 p_1f_n(s) \dd\nu_{T_+} \\
&
=
 \frac{\ga+\tilde\ga}{2T_+}\lim_{n\to+\infty} \int_0^t  \dd s\int
 \partial_{p_1}f_n(s) \dd\nu_{T_+} =0,
\end{align*}
by \eqref{eq:88t}. Using \eqref{012308-19} we obtain also
$$
 \lim_{n\to+\infty}\bbE\left[\int_0^t r_n(s) \dd s\right]=\int_0^t
 \bar \tau_+(s) \dd s.
$$
Therefore, we can rewrite \eqref{eq:timevol3} as
\begin{equation}\label{eq:timevol31}
  \begin{split}
\int_0^t \dd s\left\{ \frac{1}{\gamma}\int_0^1G''(u) r^{(n)}(s,u)\dd u- G'(1) \bar\tau_+(s)
\right\}
+o_n(t),
  \end{split}
\end{equation}
 where    
\begin{equation}
\label{brn}
\bar r^{(n)}(t,u)=\bbE \big[r_x(t)\big]\quad\mbox{ for } u\in \big[ \tfrac x n; \tfrac{x+1} n \big),
 \quad n\ge1
\end{equation} and $\lim_{n\to+\infty} \sup_{s\in[0,t]}o_n(s)=0$.
Thanks to \eqref{p-n} we know that
there exists $R>0$ such that
\begin{equation}
\label{010203}
\sup_{n\ge1}\sup_{s\in[0,t]}\big\|\bar r^{(n)}(s,\cdot)\big\|_{L^2(\T)}=:R<+\infty.
\end{equation}
The above means that   for each $s\in[0,t]$ the sequence $\left\{\bar r^{(n)}(s,\cdot)\right\}_{n\ge1}$ is contained in $\bar B_R$ -- the closed ball of radius $R>0$ in $ L^2(\T)$, centered at $0$.
 The ball  is compact in $ L^2_w(\T)$ --
the space of square integrable functions on $\T$ equipped
with the weak $ L^2$ topology.  The  topology restricted to $\bar B_R$ is metrizable.
From the above argument it follows in particular that  for each $t>0$ the sequence  $\left\{\bar r^{(n)}(\cdot)\right\}$ is equicontinuous 
 in ${\cal C}\left([0,t],\bar B_R\right)$. Thus, according to the Arzela theorem, see e.g.~\cite[p. 234]{kelley}, it is
sequentially pre-compact in the space ${\cal C}\left([0,t],L^2_w(\T)\right)$ for any $t>0$. Consequently, any limiting point of the
sequence satisfies  
\eqref{eq:wlinear}.  

\medskip

 Concerning equation \eqref{eq:4} with the respective boundary
  condition, its weak
  formulation is as follows:
for any 
$G\in L^1([0,+\infty);C_{0,1}^2(\T))$ we have
\begin{align}
\label{022308-19}
 0= & \;  
 \int_0^1 G(0,u) e(0,u)\; \dd u \notag \\ &  +
 \int_0^{+\infty} \dd s \int_0^1 \left(\partial_sG(s,u) + \frac12 (\gamma^{-1}+\gamma)
    \partial_u^2G(s,u) \right)e(s,u) \; \dd u \notag \\
  & + \frac {1}{4}(\gamma^{-1} - \gamma) \int_0^{+\infty} \dd s \int_0^1 \partial_u^2 G(s,u)
   r^2(s,u) \; \dd u \notag \\
   & -\int_0^{+\infty} \left(\partial_uG(s,1) \left[(\gamma^{-1}+\gamma) T_+
       + \frac 14 \left(\gamma^{-1} -\gamma\right) \bar\tau_+(s)^2\right]
     - \partial_uG(s,0) T_-\right) \dd s.
  \end{align}
Given a non-negative initial data $e(0,\cdot)\in L^1(\T)$ and
the macroscopic stretch $ r(\cdot)$ (determined via \eqref{eq:wlinear}) one can easily
    show that the respective weak formulation of the boundary value
    problem for a linear heat equation, resulting from \eqref{022308-19},
    admits a unique measured value solution.
 
The  averaged thermal energy  function is defined as 
$$
\bar{ \mathcal E}_n(t,u):=
\bbE \big[\mathcal E_x(t)\big],\quad u\in\big[\tfrac{x}{n},\tfrac{x+1}{n}\big),\quad x=0,\ldots,n-1.
$$
It is easy to see, thanks to \eqref{p-n}, that $\left(\bar{ \mathcal E}_n(t)\right)_{n\ge1}$
is  bounded in the dual to the separable Banach
space  $L^1([0,+\infty);C_{0,1}^2(\T))$.
Thus it is $\star$-weakly compact. In what follows we identify its limit $e(t)$
by showing  that it satisfies \eqref{022308-19}.
To achieve this goal we are going to use the microscopic energy currents given in \eqref{eq:current}. 

 Consider now a smooth test function $G\in C_0^\infty([0,+\infty)\times
\T)$  such that $G(s,0) = G(s,1) \equiv 0$, $s\ge0$.
Then,  from \eqref{eq:en-evol}, we get 
\begin{align}
   &- \frac 1n  {\mathbb E\left[\sum_{x=0}^n  G_x (0) \mathcal E_x(0) \right]} = 
        -       \frac 1n
  {\mathbb E\left[ \sum_{x=1}^{n-1} G_x (0){\mathcal E}_x(0) \right]} 
=- \frac 1n \int_0^t \dd s\; {\mathbb E\left[\sum_{x=0}^{n-1}  \partial_sG_x (s) \mathcal E_x(s) \right]}  \notag \\
    & + \int_0^t \dd s\;  {\mathbb E\left[\sum_{x=1}^{n-2} (\nabla_n G)_x(s)\; j_{x,x+1}(s) - nG_{n-1}(s) j_{n-1,n}(s)
     + nG_1(s) j_{0,1}(s) \right] }.\label{eq:cal0}
             \end{align}
Here $G_x(s):=G(s,x/n)$ and we use a likewise notation for
$\nabla_nG_x(s)$, $\Delta_nG_x(s)$.

Concerning \eqref{eq:cal0}:  by \eqref{eq:ex-2}, the expectation of its last two terms  are negligible.
In order to treat the first term of \eqref{eq:cal0},
we use the fluctuation-dissipation relation  \eqref{eq:fdst2}, i.e.
\begin{equation}
  \label{eq:fd}
  j_{x,x+1} = n^{-2} L g_x + \nabla V_x, 
\end{equation}
 with
  \[
 V_x= -\frac 1{2 \gamma}  r_x^2 - \frac{\gamma}{4}(p_x^2 + p_{x-1}^2) - \frac 1{2 \gamma} p_xp_{x-1} .
\]
Using this relation we can rewrite the last term \eqref{eq:cal0} as
\begin{equation}
  \int_0^t \dd s\; {\mathbb E}\bigg[ \frac 1{n}  \sum_{x=2}^{n-2}
                                                  (\Delta_n G)_x (s)V_x(s) \label{eq:lastb0}
                                                 -  (\nabla_n G)_{n-2} (s)V_{n-1}(s) 
                                                 +  (\nabla_n G)_1(s)  V_1(s)\bigg]  
  \end{equation}
plus expressions involving the average fluctuating term
\begin{align*}
\int_0^t \dd s  \mathbb E \bigg[ \frac 1{n^2} \sum_{x=1}^{n-2} (\nabla_n G)_x(s) & L g_x(s)\bigg] \\
= & \; \mathbb E\left[ \frac 1{n^2} \sum_{x=1}^{n-2}  \Big((\nabla_n
  G)_x(t) g_x(t) -(\nabla_n G)_x(0) g_x(0)\Big)\right]\\
&
-\int_0^t \dd s \mathbb E \left[ \frac 1{n^2} \sum_{x=1}^{n-2} (\nabla_n \partial_sG)_x(s) g_x(s)\right]
\end{align*}
which turns out to be small, as $n\to+\infty$, thanks to the energy bound \eqref{p-n}.
By Lemmas \ref{pbond} and \ref{lem:bound2} we have:
\begin{equation}
  \label{eq:6}
    \lim_{n\to\infty}  {\mathbb E}\left[ \int_0^t \dd s\; (\nabla_n G)_1(s) V_1(s) \right] =
    -  \int_0^t  \dd s \partial_u G(s,0)\frac 12\left(\gamma^{-1}+\gamma\right) T_- , 
\end{equation}
which takes care of the left boundary condition. Concerning the right
one we have
\begin{align}
  \label{eq:6aa} 
   {\mathbb E}\left[ \int_0^t \dd s\;  (\nabla_n G)_{n-2} (s)
     V_{n-1}(s) \right] = & 
- {\mathbb E}\left[ \int_0^t \dd s\;  (\nabla_n G)_{n-2} (s) \nabla V_{n-1}(s)
 \right] \\
& +  {\mathbb E}\left[ \int_0^t \dd s\;  (\nabla_n G)_{n-2} (s)
     V_{n}(s) \right].\label{eq:6ab}
  \end{align}
Using again the results of 
Lemmas \ref{pbond} and \ref{lem:bound2} we conclude that the limit of
the second term  \eqref{eq:6ab} equals
$$
     - \int_0^t  \dd s \partial_u G(s,1) \left(\frac 12\left(\gamma^{-1}+\gamma\right) T_+
     + \frac 1{2\gamma}  \bar\tau_+^2(s)\right).
$$
On the other hand, using \eqref{eq:fd}, the term \eqref{eq:6aa}
equals
\begin{equation}
  \label{eq:6b}
  \begin{split}  
   \frac{1}{n^2} {\mathbb E}\left[ \int_0^t \dd s\;  (\nabla_n
     G)_{n-2} (s) L g_{n-1}(s)
 \right]
&+  \int_0^t \dd s\;  (\nabla_n G)_{n-2} (s)
  {\mathbb E}  \big[j_{n-1,n}(s)\big].
  \end{split}
\end{equation}
From \eqref{eq:11} we conclude that the second term vanishes, with
$n\to+\infty$. By integration by parts the first term equals
\begin{multline}
  \label{eq:6c}
  \frac{1}{n^2} {\mathbb E}\Big[ (\nabla_n
  G)_{n-2} (t) g_{n-1}(t) -(\nabla_n
  G)_{n-2} (0) g_{n-1}(0)
 \Big]  \\  \qquad  -\frac{1}{n^2} {\mathbb E}\left[ \int_0^t \dd s\;  (\nabla_n
     \partial_s G)_{n-2} (s) g_{n-1}(s)
 \right],
\end{multline}
which vanishes, thanks to \eqref{p-n}. Summarizing we have shown that
\begin{multline}
  \label{eq:6d} 
 \lim_{n\to+\infty} {\mathbb E}\left[ \int_0^t \dd s\;  (\nabla_n G)_{n-2} (s)
     V_{n-1}(s) \right] \notag\\  =
   - \int_0^t  \dd s \partial_u G(s,1) \left[\frac 12\left(\gamma^{-1}+\gamma\right) T_+
     + \frac 1{2\gamma} \bar\tau_+^2(s)\right].
  \end{multline}
Now, for the bulk, it follows from \eqref{013108-19} and
  \eqref{023108-19} that
\begin{equation}
\label{033108-19}
n^{-2} L h_x = \nabla W_x-2\ga p_x p_{x-1},\quad x=2,\ldots,n-1,
\end{equation}
with
$$
h_x:=p_xp_{x-1}-\frac{r_x^2}{2},\qquad W_x:=r_{x-1} p_x.
$$
Therefore by
\eqref{eq:88t} and an argument similar to the one used above we conclude that
\begin{equation}
  \label{eq:14}
  \lim_{n\to\infty}   \int_0^t \dd s\; {\mathbb E} \left[\frac 1{n}  \sum_{x=2}^{n-2}
                                                  (\Delta_n G)_x (s) p_x(s)p_{x-1}(s) \right]= 0.
\end{equation}
In the case $\gamma = 1$ we can rewrite
\begin{equation}
  \label{eq:15}
  V_x= - \mathcal E_x  + \frac{1}{4}(p_x^2 - p_{x-1}^2) - \frac 1{2 } p_xp_{x-1} 
\end{equation}
so that it is easy to see that \eqref{eq:lastb0} is equivalent to
\begin{equation}
  \label{eq:16}
 {\color{magenta} -} \int_0^t \dd s\; {\mathbb E} \left[\frac 1{n}  \sum_{x=2}^{n-2}  (\Delta_n G)_x (s)  \mathcal E_x (s) \right]
  + o_n(1),
\end{equation}
closing the energy conservation equation and concluding the proof.

\medskip

Finally, for $\gamma \neq 1$
we expect that the following term to vanish as $n\to +\infty$:
\begin{equation}
\label{conj1}
\int_0^t \dd s \; {\mathbb E} \left[\frac 1n \sum_{x=2}^{n-2} (\Delta_n G)_x \; \Big(p_x^2(s)
    - \big(r_x(s)-\bbE \big[ r_x(s)\big]\big)^2\Big)\right],
    \end{equation}
  as can be guessed by local equilibrium considerations. 
Unfortunately in order to prove the last limit one needs some higher moment bounds that are not available
from relative entropy considerations.
One prospective work could be to proceed in an analogous way as
in the periodic case \cite{kos1},
by studying the evolution of the Wigner distribution of the thermal energy in Fourier coordinates.

\bibliographystyle{amsalpha}

\end{document}